%% file: v4.tex
\newcommand{\level}{}

\documentclass[a4paper,reqno]{amsart}
\pdfoutput=1

\title{A categorification of acyclic principal coefficient cluster algebras}
\author{Matthew Pressland}
\address{Matthew Pressland\\School of Mathematics \& Statistics\\University of Glasgow\\Glasgow G20 8LR\\United Kingdom}
\email{Matthew.Pressland@glasgow.ac.uk}

\input{\level options-std.tex}

\input{\level amsthm-std.tex}
\input{\level biblatex-std.tex}

\renewcommand{\env}[1]{{#1}^{\varepsilon}}
\newcommand{\FKcc}[1]{\varphi^{#1}}
\newcommand{\frozen}{dashed}
\renewcommand{\Head}[1]{\mathrm{H}_{#1}}
	\newcommand{\mHead}[1]{\mathrm{H}_{#1}^{\mut}}

\newcommand{\lift}[1]{\widetilde{#1}}
\newcommand{\mut}{\circ}
\newcommand{\Palcc}[1]{\underline{\varphi}^{#1}}
\newcommand{\princlust}[1]{\clustalg{#1}^\bullet}
\newcommand{\pprinclust}[1]{\widetilde{\mathscr{A}}_{#1}}
\newcommand{\rel}[1]{\rho_{#1}}
\newcommand{\res}[1]{\mathbf{P}(#1)}
\newcommand{\rhom}[1]{\Omega_{#1}}
\renewcommand{\Tail}[1]{\mathrm{T}_{#1}}
	\newcommand{\mTail}[1]{\mathrm{T}_{#1}^{\mut}}
\newcommand{\vin}[1]{\mathrm{H}_v}
\newcommand{\vout}[1]{\mathrm{T}_v}
\usepackage[T1]{fontenc}

\date{15th February 2023}

\begin{document}

\begin{abstract}
In earlier work, the author introduced a method for constructing a Frobenius categorification of a cluster algebra with frozen variables by starting from the data of an internally Calabi--Yau algebra, which becomes the endomorphism algebra of a cluster-tilting object in the resulting category. In this paper, we construct appropriate internally Calabi--Yau algebras for cluster algebras with polarised principal coefficients (which differ from those with principal coefficients by the addition of more frozen variables), and obtain Frobenius categorifications in the acyclic case. Via partial stabilisation, we then define extriangulated categories, in the sense of Nakaoka and Palu, categorifying acyclic principal coefficient cluster algebras, for which Frobenius categorifications do not exist in general. Many of the intermediate results used to obtain these categorifications remain valid without the acyclicity assumption, as we will indicate, and are interesting in their own right. Most notably, we provide a Frobenius version of Van den Bergh's result that the Ginzburg dg-algebra of a quiver with potential is bimodule $3$-Calabi--Yau.
\end{abstract}
\maketitle

\section{Introduction}

Cluster algebras, introduced by Fomin--Zelevinsky \cite{FZ-CA1}, are combinatorially defined algebras with applications to many areas of mathematics, and have been the subject of intense study; see Keller \cite{kellercluster} for a survey of connections between cluster algebras and the representation theory of associative algebras, and the references therein for applications to other fields.

A cluster algebra is typically defined using a generating set of cluster variables, which is constructed iteratively  from the data of an initial seed via a process of mutation. Here we consider cluster algebras of geometric type coming from quivers (equivalently, skew-symmetric matrices), so that a seed is given by the data of a quiver together with a collection of rational functions (which are some of the aforementioned cluster variables) in bijection with its vertices. Mutating at one of the quiver vertices modifies the quiver according to Fomin--Zelevinsky's mutation rule (see for example \cite[\S3.2]{kellercluster}), and exchanges the cluster variable attached to this vertex for a new one. A subset of the quiver vertices may be frozen (giving the quiver the structure of an ice quiver, see Definition~\ref{d:frjacalg}), indicating that mutation at these vertices is not permitted, and thus the corresponding cluster variables, also called frozen, appear in every seed. 

Categorification provides an effective tool for studying this combinatorics via the representation theory of quivers. Cluster algebras without frozen variables have been successfully categorified in large generality, beginning with Buan--Marsh--Reineke--Reiten--Todorov's construction \cite{BMRRT} of cluster categories of acyclic quivers, later generalised by Amiot \cite{Amiot-ClustCat} to allow for the existence of cycles. In these categories, certain maximal rigid objects, called cluster-tilting objects, play the role of seeds---in particular, they can be mutated at their indecomposable summands \cite{IY-mutation}. Cluster algebras categorified in this way have seeds in bijection with cluster-tilting objects (within a fixed mutation class) in the category, and cluster variables in bijection with indecomposable rigid objects (appearing as indecomposable summands of cluster-tilting objects in the fixed mutation class). Cluster categories of acyclic quivers have only a single mutation class of cluster-tilting objects \cite[Thm.~A.1]{BMRT}, and so the caveats in the preceding sentence may be dropped.

These results have led to clean, conceptual proofs of many key statements for those cluster algebras without frozen variables admitting such a categorification, including `cluster determines seed' \cite{BMRT}, linear independence of cluster monomials \cite{cerullilinear}, sign coherence of c-vectors \cite{speyeracyclic}, and so on. However, most cluster algebras appearing in nature, such as the cluster structures on the coordinate rings of partial flag varieties \cite{GLS-PFVs}, their double Bruhat cells \cite{berensteincluster3}, and Grassmannian positroid strata \cite{GL-PosVars}, do have frozen variables, which we would also like to capture in a categorification.

This has been achieved for a number of families of cluster algebras \cite{demoneticequivers1, GLS-PFVs, JKS, NC-2CY, DI} using suitable stably $2$-Calabi--Yau Frobenius categories in place of the $2$-Calabi--Yau triangulated cluster categories appearing in the case of no frozen variables. A Frobenius category is, by definition, an exact category with enough projective objects and enough injective objects, such that these two classes of objects coincide. It is the indecomposable projective-injective objects, which necessarily appear as summands of any cluster-tilting object, that will correspond to the frozen variables. The fact that specialising all frozen variables to $1$ in a cluster algebra produces a new cluster algebra without frozen variables (which we call the principal part) corresponds to the fact that the stable category of a Frobenius category, defined as the quotient by the ideal of morphisms factoring through a projective object (so that these objects are isomorphic to the zero object in the quotient), is a triangulated category \cite[\S I.2]{happeltriangulated}. We require this stable category to be $2$-Calabi--Yau so that cluster-tilting objects may be mutated in the appropriate way.

In this paper, we will primarily consider the cluster algebra $\pprinclust{Q}$ with polarised principal coefficients associated to a quiver $Q$. We define this cluster algebra precisely in Section~\ref{s:pp-coeffs}, but for now note that each of its seeds has $3n$ cluster variables, of which $2n$ are frozen, when $Q$ has $n$ vertices. Specialising half of these frozen variables to $1$, we obtain the principal coefficient cluster algebra $\princlust{Q}$ attached to $Q$ by Fomin--Zelevinsky. They used this cluster algebra to study the combinatorics of other cluster algebras having the same principal part; for example, their expansion formula \cite[Cor.~6.3]{FZ-CA4} for cluster variables in such a cluster algebra involves F-polynomials, defined as specialisations of cluster variables in $\princlust{Q}$ \cite[Eq.~3.3]{FZ-CA4}, and hence also obtained as specialisations of cluster variables in $\pprinclust{Q}$. Many of the results of Gross--Hacking--Keel--Kontsevich \cite{grosscanonical} on canonical bases are proved by first reducing to the case of principal coefficients. Further applications of principal coefficient cluster algebras include Lee--Schiffler's realisation \cite{leecluster} of Jones polynomials of $2$-bridge links as specialisations of their cluster variables. The polarised principal coefficient cluster algebra $\pprinclust{Q}$, in the case that $Q$ is a Dynkin quiver, has itself appeared in recent work of Borges and Pierin \cite{borgescluster}, who define a modified cluster character on the ordinary triangulated cluster category of $Q$, taking values in $\pprinclust{Q}$. (A more general version of this final result is a consequence of our considerations here, as explained in Remark~\ref{r:borges-pierin}.)

Our main result is to construct a Frobenius categorification of the cluster algebra $\pprinclust{Q}$ in the case that $Q$ is acyclic.

\begin{thm*}[Categorification for polarised principal coefficients]
\label{t:main-cat-thm}
Let $Q$ be an acyclic quiver. Then the category $\frobcat_Q$ (Definition~\ref{d:pprin-cat})
\begin{enumerate}
\item\label{t:main-cat-thm-fcc}is a Hom-finite Frobenius cluster category (Definition~\ref{d:frobclustcat}),
\item\label{t:main-cat-thm-cc} has stable category $\stab{\frobcat}_Q$ equivalent to the cluster category $\clustcat{Q}$,
\item\label{t:main-cat-thm-clustvars} carries a Fu--Keller cluster character \cite{FuKel}, inducing a bijection between isomorphism classes of indecomposable rigid objects of $\frobcat_Q$ and cluster variables of $\pprinclust{Q}$, and
\item\label{t:main-cat-thm-mut} the previous bijection induces a further bijection between isomorphism classes of cluster-tilting objects of $\frobcat_Q$ and seeds of the polarised principal coefficient cluster algebra $\pprinclust{Q}$, commuting with mutation, such that the ice quiver of the endomorphism algebra\footnote{Throughout the paper, we interpret the quivers of endomorphism algebras as ice quivers by declaring the frozen vertices to be those corresponding to projective indecomposable summands.} of each cluster-tilting object agrees, up to arrows between frozen vertices, with the ice quiver of the corresponding seed.
\end{enumerate}
\end{thm*}

In addition to providing categorifications for a new family of cluster algebras, the proof of Theorem~\ref{t:main-cat-thm} demonstrates the effectiveness of a new and general methodology for making such constructions, based on prior work of the author \cite{Pressland-iCY}. Earlier work providing Frobenius categorifications of cluster algebras has typically depended on having some insight into the cluster algebra in its totality, rather than just the data of an initial seed, before constructing the categorification. For example, Geiß--Leclerc--Schröer \cite{GLS-PFVs}, Demonet--Luo \cite{demoneticequivers1}, Jensen--King--Su \cite{JKS} and Demonet--Iyama \cite{DI} construct Frobenius categorifications for cluster algebra structures on coordinate rings of partial flag varieties by exploiting the fact that these rings are well-understood geometrically. The Frobenius categorifications of universal coefficient cluster algebras by Nájera Chávez \cite{NC-2CY} are restricted to finite type, again making the global combinatorics of the cluster algebra (such as its exchange graph, cluster complex, and so on) more tractable. By contrast, our approach is more akin to the work of Buan--Marsh--Reineke--Reiten--Todorov \cite{BMRRT} and Amiot \cite{Amiot-ClustCat}, in which categorifications are constructed from the local data defining the cluster algebra, namely the initial seed (enhanced in Amiot's case by choosing the extra data of a potential on the quiver). 

Indeed, our construction begins with the definition of an algebra $A$ as a particular quotient of the (complete) path algebra of a quiver $\lift{Q}$ agreeing with that of the initial seed of $\pprinclust{Q}$ up to the addition of arrows between frozen vertices. The algebra $A$ is finite-dimensional and has a strong homological symmetry property---it is bimodule internally $3$-Calabi--Yau with respect to the frozen vertices \cite[Def.~2.4]{Pressland-iCY}. We then construct the category $\frobcat_Q$ from $A$ by applying earlier results of the author \cite[Thm.~4.1, Thm.~4.10]{Pressland-iCY}. The main step in the proof of Theorem~\ref{t:main-cat-thm}, taking up Sections~\ref{s:pp-coeffs}--\ref{s:acyclic} of the paper, is to define $A$ and show that it satisfies the assumptions of these theorems.

As a consequence of this approach we obtain the following surprising result, which is not specific to acyclic quivers, and which we expect to be of independent interest. For the definitions of ordinary and frozen Jacobian algebras, see Definition~\ref{d:frjacalg}.

\begin{thm*}[Corollary of Theorem~\ref{t:bi3cy}]
\label{t:keller-analogue}
Let $(Q,W)$ be a quiver with potential, and let $\stab{A}=\jac{Q}{W}$ be its Jacobian algebra. Then there is a frozen Jacobian algebra $A=A_{Q,W}$ (Definition~\ref{d:bigQP}), internally bimodule $3$-Calabi--Yau with respect to its frozen idempotent $e$, such that $\stab{A}=A/\Span{e}$.
\end{thm*}

We see this statement as analogous to a result of Van den Bergh \cite[Thm.~A.12]{kellerdeformed}, implying that the Jacobian algebra of a quiver with potential may be realised as the $0$-th homology of a bimodule $3$-Calabi--Yau dg-algebra constructed by Ginzburg \cite{ginzburgcalabiyau}. Indeed, the construction of $A_{Q,W}$ has some features in common with that of the Ginzburg dg-algebra $\Gamma_{Q,W}$ of $(Q,W)$ (see Remark~\ref{r:Ginzburg}), but at present we do not know a precise recipe for constructing $A_{Q,W}$ from $\Gamma_{Q,W}$.

It is necessary to pass from the more familiar system of principal coefficients to some larger collection of frozen variables in order for a result like Theorem~\ref{t:main-cat-thm} to be possible, since a Frobenius categorification of the principal coefficient cluster algebra $\princlust{Q}$ cannot exist; see Proposition~\ref{p:no-prin-cat}. However, by taking a suitable quotient of $\frobcat_Q$, we also define a second category $\frobcat^+_Q$, which is extriangulated in the sense of Nakaoka and Palu \cite{nakaokaextriangulated}. This second category is the categorification of $\princlust{Q}$ referred to in the title. Our construction allows the desired properties of $\frobcat^+_Q$ to be deduced directly from the corresponding properties of $\frobcat_Q$ appearing in Theorem~\ref{t:main-cat-thm}, via \cite[Prop.~3.30]{nakaokaextriangulated}.

\begin{cor*}[Categorification for principal coefficients]
\label{c:extriangulated}
Let $Q$ be an acyclic quiver. Then the category $\frobcat^+_Q$ (Definition~\ref{d:prin-cat}) has the properties that
\begin{enumerate}
\item\label{c:extriangulated-cc} the stable category $\stab{\frobcat}^+_Q$ is equivalent to the cluster category $\clustcat{Q}$,
\item\label{c:extriangulated-clustvars} the cluster character on $\frobcat_Q$ (see Theorem~\ref{t:main-cat-thm}\ref{t:main-cat-thm-clustvars}) induces a bijection between isomorphism classes of indecomposable rigid objects of $\frobcat^+_Q$ and cluster variables of $\princlust{Q}$, and
\item\label{c:extriangulated-mut} this induces a further bijection between  isomorphism classes of cluster-tilting objects of $\frobcat_Q$ and seeds of the principal coefficient cluster algebra $\pprinclust{Q}$, commuting with mutation, such that the ice quiver of the endomorphism algebra of each cluster-tilting object agrees, up to arrows between frozen vertices, with the ice quiver of the corresponding seed.
\end{enumerate}
\end{cor*}


We note that Fu--Keller \cite{FuKel} have a different approach to categorifying $\princlust{Q}$, via a full subcategory $\mathcal{U}_Q$ of the cluster category $\clustcat{Q^\bullet}$ attached to the quiver $Q^\bullet$ of the standard initial seed of $\princlust{Q}$. Since it is extension-closed as a subcategory of the triangulated category $\clustcat{Q^\bullet}$, the category $\mathcal{U}_Q$ is naturally extriangulated \cite[Rem.~2.18]{nakaokaextriangulated}.

\begin{conj}
\label{conj:P-vs-FK}
The category $\frobcat^+_Q$ from Corollary~\ref{c:extriangulated} is equivalent, as an extriangulated category, to Fu--Keller's category $\mathcal{U}_Q$.
\end{conj}

Producing equivalences between `cluster-like' categories is known to be a hard problem, although there are some results of this kind for triangulated categories, due to Keller--Reiten \cite{kelleracyclic} and Amiot--Reiten--Todorov \cite{amiotubiquity}. At present, we are not aware of any general result constructing such equivalences between extriangulated categories. Even if Conjecture~\ref{conj:P-vs-FK} is true, the fact that we describe the categorification via a partial stabilisation of the Frobenius category $\frobcat_Q$ instead of as an extension closed subcategory of $\clustcat{Q^\bullet}$ gives us access to different techniques for studying it, as we demonstrate in Section~\ref{s:index}. Indeed, cluster theory is at present better developed for exact categories than for extriangulated categories such as $\frobcat_Q^+$ and $\mathcal{U}_Q$, and it is useful to be able to exploit this.

The main results of the paper are contained in Sections~\ref{s:pp-coeffs}--\ref{s:extriangulated}. In Section~\ref{s:pp-coeffs} we describe the cluster algebras with polarised principal coefficients that we will categorify, and recall the results of \cite{Pressland-iCY}, which we will use to construct the category $\frobcat_Q$ appearing in Theorem~\ref{t:main-cat-thm}. The algebra $A=A_{Q,W}$ needed as input for this construction is defined in Section~\ref{s:bigQP}, from the data of a quiver with potential $(Q,W)$. In Section~\ref{s:frjacCY} we explain results from \cite{Pressland-iCY} which allow us to establish the bimodule internally $3$-Calabi--Yau property for a frozen Jacobian algebra, and apply these results to $A$, thus proving Theorem~\ref{t:keller-analogue}. We first restrict to acyclic quivers in Section~\ref{s:acyclic}, where we show that $A$ is finite-dimensional, and hence Noetherian, under this assumption---this is why we require $Q$ to be acyclic in Theorem~\ref{t:main-cat-thm}. All of these ingredients are combined in Section~\ref{s:mutation} to give a proof of Theorem~\ref{t:main-cat-thm}. The proof of Corollary~\ref{c:extriangulated} then follows in Section~\ref{s:extriangulated}.

In the remaining sections we investigate some of the features of the category $\frobcat_Q$ appearing in Theorem~\ref{t:main-cat-thm}. This category is defined as the category of Gorenstein projective modules over an Iwanaga--Gorenstein algebra $B_Q$, which we describe explicitly via a quiver with relations in Section~\ref{s:bdy-algs}. In Section~\ref{s:index} we explain how to use the category $\frobcat_Q$ to recover cluster-algebraic information about $\pprinclust{Q}$ and $\princlust{Q}$. We close in Section~\ref{s:egs} with examples, in particular observing that Theorem~\ref{t:main-cat-thm} remains true when $Q$ is a $3$-cycle, provided we replace $\clustcat{Q}$ by Amiot's cluster category $\clustcat{Q,W}$ for a particular potential $W$.

Throughout, algebras are assumed to be associative and unital. All modules are left modules, the composition of maps $f\colon X\to Y$ and $g\colon Y\to Z$ is denoted by $gf$, and we use the same convention for compositions of arrows in quivers. The Jacobson radical of a module $X$ is denoted by $\rad{X}$. If $p$ is a path in a quiver, we denote its head by $\head{p}$ and its tail by $\tail{p}$. We fix an algebraically closed field $\KK$, over which all our algebras and categories will be defined unless specified otherwise. For a $\KK$-algebra $A$, we denote the module category of $A$ by $\Modcat{A}$ and the subcategory of finitely generated $A$-modules by $\fgmod{A}$.

\section{Coefficient systems and cluster categories}
\label{s:pp-coeffs}
In this section we introduce our main definitions. We will not give the full definition of a cluster algebra, since this is somewhat lengthy and can be found easily in many other sources; we recommend Keller's survey \cite{kellercluster}. For concreteness, our cluster algebras will be defined over $\QQ$, and we do not assume frozen variables to be invertible; that is, the cluster algebra $\clustalg{}$ generated by an initial seed with cluster $(x_1,\dotsc,x_n,x_{n+1},\dotsc,x_m)$, in which the mutable variables are $x_1,\dotsc,x_n$, is defined to be a subalgebra of $\QQ(x_1,\dotsc,x_n)[x_{n+1},\dotsc,x_m]$. However, since our results are concerned with the set of cluster variables of $\clustalg{}$, their grouping into clusters, and the exchange graph on these clusters, they are insensitive to changing these conventions by replacing $\QQ$ by a field extension or inverting the frozen variables. As such, we de-emphasise these conventions below, and simply refer to `the' cluster algebra generated by a seed.

Let $\clustalg{}$ be a cluster algebra of geometric type without frozen variables, and let $s_0$ be a seed of $\clustalg{}$, with quiver $Q$ and cluster variables $(x_1,\dotsc,x_n)$. By definition, the quiver $Q$ has no loops or $2$-cycles. The principal coefficient cluster algebra $\princlust{Q}$ corresponding to this data is defined by the following initial seed. The mutable cluster variables are again $(x_1,\dotsc,x_n)$, and the frozen variables are $(x_1^+,\dotsc,x_n^+)$, where the indexing reveals a preferred bijection between the mutable and frozen variables. 
The ice quiver $Q^\bullet$ of this seed contains $Q$ as a full subquiver, with mutable vertices, and for each vertex $i\in Q_0$ (corresponding to the variable $x_i$), $Q^\bullet$ has a frozen vertex $i^+$ (corresponding to $x_i^+$) and an arrow $i\to i^+$. While $\clustalg{}$ is isomorphic (as a cluster algebra) to the cluster algebra determined by any quiver mutation equivalent to $Q$, this is not true of $\princlust{Q}$.

Specialising all frozen variables of $\princlust{Q}$ to $1$ to obtain a cluster algebra without frozen variables (which we refer to as taking the \emph{principal part}) recovers $\clustalg{}$, and gives a bijection between the seeds of $\princlust{Q}$ and those of $\clustalg{}$ \cite{cerullilinear}; we write $s^\bullet$ for the seed of $\princlust{Q}$ corresponding to a seed $s$ of $\clustalg{}$. Precisely, if $s^\bullet$ has cluster $(v_1,\dotsc,v_n,x_1^+,\dotsc,x_n^+)$, then $s$ has cluster $(v_1,\dotsc,v_n)|_{x_i^+=1,1\leq i\leq n}$, and the quiver of $s$ is the full subquiver of that of $s^\bullet$ on the mutable vertices. Principal coefficients are important since knowledge of the cluster algebra $\princlust{Q}$ can be used to study any cluster algebra $\clustalg{}'$ with principal part $\clustalg{}$, via the theory of g-vectors and F-polynomials \cite{FZ-CA4}.

By choosing some extra data on $Q$, Amiot \cite{Amiot-ClustCat} constructs a categorification of $\clustalg{}$. The construction uses Jacobian algebras, so we recall some relevant definitions, at the level of generality needed later in the paper.

\begin{defn}
\label{d:frjacalg}
An \emph{ice quiver} $(Q,F)$ consists of a finite quiver $Q$ without loops and a subquiver $F$ of $Q$. Denote by $\cpa{\KK}{Q}$ the completion of the path algebra of $Q$ over our fixed algebraically closed field $\KK$ with respect to the ideal $J_Q$ generated by arrows, treated as a topological algebra with the $J_Q$-adic topology. A \emph{potential} on $Q$ is an element $W$ of the vector space quotient  $\cpa{\KK}{Q}/\close{\{\cpa{\KK}{Q},\cpa{\KK}{Q}\}}$, of $\cpa{\KK}{Q}$ by the closure of the subspace spanned by commutators, such that $W$ lies in the image of the natural map from $J_Q$ to this vector space. (More informally, $W$ is a possibly infinite linear combination of cyclic equivalence classes of cycles of positive length in $Q$.)

A vertex or arrow of $Q$ is called \emph{frozen} if it is a vertex or arrow of $F$, and \emph{mutable} or \emph{unfrozen} otherwise. For brevity, we write $Q_0^\mut=Q_0\setminus F_0$ and $Q_1^\mut=Q_1\setminus F_1$ for the sets of mutable vertices and unfrozen arrows respectively. For $\alpha\in Q_1$ and $\alpha_k\dotsm\alpha_1$ a cycle in $Q$, write
\[\der{\alpha}{\alpha_k\dotsm\alpha_1}=\sum_{\alpha_i=\alpha}\alpha_{i-1}\dotsm\alpha_1\alpha_k\dotsm\alpha_{i+1}.\]
Extending the map $\der{\alpha}$ linearly and continuously, and noting that it vanishes on commutators, allows us to define $\der{\alpha}{W}$. The ideal $\Span{\der{\alpha}{W}:\alpha\in Q_1^\mut}$ of $\cpa{\KK}{Q}$ is called the \emph{Jacobian ideal}, and we define the \emph{frozen Jacobian algebra} associated to $(Q,F,W)$ by
\[\frjac{Q}{F}{W}=\cpa{\KK}{Q}/\close{\Span{\der{\alpha}{W}:\alpha\in Q_1^\mut}}.\]
Write $A=\frjac{Q}{F}{W}$. The above presentation of $A$ suggests a preferred idempotent
\[e=\sum_{v\in F_0}\idemp{v},\]
which we call the \emph{frozen idempotent}. We will call the subalgebra $B=eAe$ the \emph{boundary algebra} of $A$. If $F=\varnothing$, then we refer to the pair $(Q,W)$ as a quiver with potential, and call $\jac{Q}{W}=\frjac{Q}{\varnothing}{W}$ simply the Jacobian algebra of $(Q,W)$. If $\jac{Q}{W}$ is finite-dimensional, we call both $W$ and the pair $(Q,W)$ \emph{Jacobi-finite}.
\end{defn}

\begin{rem}
\label{r:completion}
In much of the paper, we will consider algebras defined by a quiver $Q$ together with a set of relations $R$, in the ordinary path algebra $\KK Q$, having the property that the natural map $\KK Q/\Span{R}\to\cpa{\KK}{Q}/\close{\Span{R}}$ is an isomorphism---for example, this happens when the domain is finite-dimensional. However, we take complete path algebras in the general theory since we want the categories we construct to be Krull--Schmidt \cite[\S4]{krausekrullschmidt} (cf.\ \cite[Rem.~3.3]{JKS}). Further, we note that the Jacobson radical $J_Q$ of $\cpa{\KK}{Q}$ is generated by the arrows (unlike the Jacobson radical of $\KK Q$ when $Q$ has cycles). Since the Jacobian ideal defining $A=\frjac{Q}{F}{W}$ is contained in $J_Q$, it follows that $A$ has Jacobson radical $\Jacrad{A}$ generated by the arrows of $Q$, and so $A/\Jacrad{A}\cong\cpa{\KK}{Q}/J_Q\cong\KK^{Q_0}$ is a semi-simple algebra. We will use this fact in Section~\ref{s:frjacCY}.

The terminology `boundary algebra' is inspired by and compatible with \cite{BKM} (see also \cite{Pressland-Post}), although in the general case $Q$ need not admit any embedding into a manifold so that $F$ is embedded into the boundary.
\end{rem}

Assume that $Q$ is a quiver admitting a Jacobi-finite potential $W$. Then by work of Amiot \cite{Amiot-ClustCat}, there is a $2$-Calabi--Yau triangulated category $\clustcat{Q,W}$ categorifying $\clustalg{Q}$. If $W$ is additionally non-degenerate \cite[Def.~7.2]{derksenquivers1}, then the seeds of $\clustalg{}$ correspond bijectively to additive equivalence classes of cluster-tilting objects of $\clustcat{Q,W}$ related by a finite sequence of mutations to an initial cluster-tilting object $T_0$ with $\Endalg{\clustcat{Q,W}}{T_0}=\jac{Q}{W}$; such cluster-tilting objects are called \emph{reachable}, and we denote the seed corresponding to such an object $T$ by $s_T$. Non-degeneracy is needed here to see, using results of Buan--Iyama--Reiten--Smith \cite{BIRS2}, that mutation of cluster-tilting objects in the mutation class of $T_0$ corresponds to Fomin--Zelevinsky mutation at the level of quivers of endomorphism algebras, and so in particular mutation is well-defined at any indecomposable summand of any reachable cluster-tilting object.

A priori, to categorify the principal coefficient cluster algebra $\princlust{Q}$, we would like to find a Frobenius category $\frobcat$ such that
\begin{enumerate}
\item\label{i:prin-cat-cc} the stable category $\stab{\frobcat}$ is triangle equivalent to $\clustcat{Q,W}$ (so that cluster-tilting objects of $\stab{\frobcat}$ may be identified with those of $\clustcat{Q,W}$, and hence cluster-tilting objects in $\frobcat$ can be mutated at their non-projective indecomposable summands), and
\item\label{i:prin-cat-mut} for any reachable cluster-tilting object $T$ there is an isomorphism $\cpa{\KK}{Q'}/\close{I}\isoto\Endalg{\frobcat}{T}$, for some ideal $I\subseteq J_{Q'}^2$, where $Q'$ is, up to arrows between frozen vertices, the quiver of the seed $s_T^\bullet$ of $\princlust{Q}$. (As always, we define the frozen vertices of $Q'$ to be those whose image under the given isomorphism to $\Endalg{\frobcat}{T}$ is a projection onto an indecomposable projective summand.)
\end{enumerate}
Unfortunately, this is not possible.

\begin{prop}
\label{p:no-prin-cat}
There does not exist a Frobenius category $\frobcat$ satisfying conditions \ref{i:prin-cat-cc} and \ref{i:prin-cat-mut} above.
\end{prop}
\begin{proof}
Assume $\frobcat$ were such a Frobenius category, and let $\hat{T}_0$ be the cluster-tilting object in $\frobcat$ which, when viewed as an object of $\stab{\frobcat}$, is identified with $T_0$ under the equivalence in \ref{i:prin-cat-cc}. By \ref{i:prin-cat-mut}, we have an isomorphism $\Endalg{\frobcat}{\hat{T}_0}\isoto\cpa{\KK}{Q}/\close{I}$ for $I\subseteq J_Q^2$ and $Q$ a quiver differing from $Q^\bullet$ only by the possible addition of arrows between frozen vertices. This latter property means that there is no path in $Q$ from a frozen vertex to a mutable one, and hence there are no non-zero maps in $\frobcat$ from an indecomposable non-projective summand of $\hat{T}_0$ to an injective object in $\frobcat$, contradicting the assumption that $\frobcat$ has enough injective objects.
\end{proof}

For this reason, we will first consider, and categorify, a different cluster algebra $\pprinclust{Q}$ with initial seed obtained from that of $\princlust{Q}$ by adding more frozen variables. Starting from our seed $s_0$ of $\clustalg{}$, with quiver $Q$ and cluster variables $(x_1,\dotsc,x_n)$, we construct an initial seed $\lift{s}_0$ of $\pprinclust{Q}$ as follows. The seed $\lift{s}_0$ has mutable variables $(x_1,\dotsc,x_n)$, and frozen variables $(x_1^+,\dotsc,x_n^+,x_1^-,\dotsc,x_n^-)$. Its ice quiver $\lift{Q}$ will be described fully in Definition~\ref{d:bigQP}; the relevant features for defining the cluster algebra are that $\lift{Q}$ contains $Q$ as a full subquiver, with mutable vertices, and has two frozen vertices $i^+$ (corresponding to $x_i^+$) and $i^-$ (corresponding to $x_i^-)$ for each mutable vertex $i\in Q_0$, with arrows $i\to i^+$ and $i^-\to i$. (In Definition~\ref{d:bigQP} we also describe arrows between the frozen vertices, that will play a role in our categorification, but are not important for defining $\pprinclust{Q}$.) We call $\pprinclust{Q}$ the \emph{polarised principal coefficient cluster algebra} associated to $Q$.

We adopt the word `polarised', referring to the partitioning of the frozen variables into two sets, to differentiate this coefficient system from the double principal coefficients studied by Rupel--Stella--Williams \cite{rupelgeneralized}. Since one encounters the same issues categorifying cluster algebras with double principal coefficients as in the case of ordinary principal coefficients, namely that no categorification may have enough injective objects, our preference here is for the polarised version.

As discussed in the introduction, we will construct a categorification of $\pprinclust{Q}$ using methodology introduced by the author in \cite{Pressland-iCY}. We now recall the key definitions and results needed to explain this construction. Given a $\KK$-algebra $A$, we write $\dcat(A)$ for its derived category and $\env{A}=A\tensor_\KK A^{\op}$ for its enveloping algebra, modules of which are precisely $A$-bimodules. We denote by $\per{A}$ the perfect derived category of $A$, i.e.\ the thick subcategory of $\dcat(A)$ generated by $A$, and write $\rhom{A}=\RHom_{\env{A}}(A,\env{A})$.

\begin{defn}[{\cite[Def.~2.4]{Pressland-iCY}}]
\label{d:bi3cy}
Given an algebra $A$ and idempotent $e\in A$, we say $A$ is \emph{bimodule internally $3$-Calabi--Yau} with respect to $e$ if
\begin{enumerate}
\item $\pdim_{\env{A}}{A}\leq 3$,
\item $A\in\per{\env{A}}$, and
\item there exists a triangle $A\map{\psi}\rhom{A}[3]\map{}C\map{}A[1]$ in $\dcat(\env{A})$, such that
\[\RHom_A(C,M)=0=\RHom_{A^{\op}}(C,N)\]
for any $M\in\dcat(A)$ whose total cohomology is a finite-dimensional $A/AeA$-module, and any $N\in\dcat(A^{\op})$ satisfying the analogous condition.
\end{enumerate}
\end{defn}

\begin{rem}
\label{r:i3cy}
Assume $A$ is bimodule internally $3$-Calabi--Yau with respect to $e$. Then $\gldim{A}\leq 3$, and there is a functorial duality
\[\Kdual\Ext^i_A(M,N)=\Ext^{3-i}_A(N,M)\]
for any finite-dimensional $A/AeA$-module $M$ and any $A$-module $N$ \cite[Cor.~2.9]{Pressland-iCY}. Moreover, $A^{\op}$ has the same properties \cite[Rem.~2.6]{Pressland-iCY}. 
\end{rem}

To construct our Frobenius categories, we will use the following theorem.

\begin{thm}[{\cite[Thm.~4.1, Thm.~4.10]{Pressland-iCY}}]
\label{t:icytofrobcat}
Let $A$ be an algebra, and $e\in A$ an idempotent. If $A$ is Noetherian, $\stab{A}=A/AeA$ is finite-dimensional, and $A$ is bimodule internally $3$-Calabi--Yau with respect to $e$, then
\begin{enumerate}
\item\label{t:icytofrobcat-frob} $B=eAe$ is Iwanaga--Gorenstein (of injective dimension at most $3$), so the category
\[\GP(B)=\set{X\in\fgmod{B}:\text{$\Ext^i_B(X,B)=0$ for all $i>0$}}\]
of Gorenstein projective $B$-modules is a Frobenius category,
\item\label{t:icytofrobcat-ct} $eA\in\GP(B)$ is cluster-tilting,
\item\label{t:icytofrobcat-isos} there are natural isomorphisms $\Endalg{B}{eA}\isoto A$ and $\stabEndalg{B}{eA}\isoto\stab{A}$, and
\item\label{t:icytofrobcat-cy} the stable category $\stabGP(B)$ is $2$-Calabi--Yau.
\end{enumerate}
\end{thm}

In many cases---in particular in our application in the proof of Theorem~\ref{t:main-cat-thm}---the Frobenius category $\GP(B)$ appearing in Theorem~\ref{t:icytofrobcat} is a Frobenius cluster category, as studied in \cite{Pressland-iCY}, and defined as follows.

\begin{defn}[cf.~{\cite[Def.~3.3]{Pressland-iCY}}]
\label{d:frobclustcat}
An Frobenius category $\frobcat$ is called a Frobenius cluster category if it is Krull--Schmidt, stably $2$-Calabi--Yau, and $\gldim{\Endalg{\frobcat}{T}}\leq3$ for any cluster-tilting object $T\in\frobcat$, of which there is at least one.
\end{defn}

\begin{prop}
\label{p:frob-clust-cat}
In the context of Theorem~\ref{t:icytofrobcat}, if $A$ is additionally finite-dimensional, then $\GP(B)$ is a Frobenius cluster category.
\end{prop}
\begin{proof}
The category $\GP(B)$ is idempotent complete for any algebra $B$, Frobenius by Theorem~\ref{t:icytofrobcat}\ref{t:icytofrobcat-frob}, and stably $2$-Calabi--Yau by Theorem~\ref{t:icytofrobcat}\ref{t:icytofrobcat-cy}. Now we use the extra assumption that $A$ is finite-dimensional. First, by \cite[Prop.~3.7]{Pressland-iCY}, the endomorphism algebra of a cluster-tilting object in $\GP(B)$ has global dimension at most $3$ provided it is Noetherian; when $A$ is finite-dimensional, the category $\GP(B)$ is Hom-finite, so this condition certainly holds. Moreover, finite-dimensional algebras are semi-perfect, which together with the earlier observation that $\GP(B)$ is idempotent complete implies that it is a Krull--Schmidt category by \cite[Cor.~4.4]{krausekrullschmidt}. 
\end{proof}

\begin{rem}
As suggested by the proof, the assumption that $A$ is finite-dimensional is likely to be much stronger than needed for the conclusion of Proposition~\ref{p:frob-clust-cat}. For example, it may already be true under the assumptions of Theorem~\ref{t:icytofrobcat} that all cluster-tilting objects in $\GP(B)$ have Noetherian endomorphism algebra---this is true for the cluster-tilting object $eA$ by the theorem, but we do not currently know how to deduce it for other cluster-tilting objects. Moreover, there are many infinite-dimensional semi-perfect algebras, such as complete path algebras of finite quivers with cycles.
\end{rem}

Returning to the problem of categorifying the cluster algebra $\pprinclust{Q}$, our aim now is to construct an algebra $A$ satisfying the conditions of Theorem~\ref{t:icytofrobcat}, defined as a quotient of the complete path algebra of the quiver $\lift{Q}$ used to define $\pprinclust{Q}$.

\section{An ice quiver with potential}
\label{s:bigQP}

Consider again our initial seed $s_0$ for $\clustalg{}$, with quiver $Q$, and choose a potential $W$ on $Q$. In this section, we will construct from $(Q,W)$ an ice quiver with potential $(\lift{Q},\lift{F},\lift{W})$, and thus a frozen Jacobian algebra $\frjac{\lift{Q}}{\lift{F}}{\lift{W}}$. It is this algebra that we intend to use as the input for the construction of a Frobenius category by Theorem~\ref{t:icytofrobcat}; to apply this theorem we will require that $\jac{Q}{W}$ is finite-dimensional, but we do not assume this yet.

\begin{defn}
\label{d:bigQP}
Let $(Q,W)$ be a quiver with potential. We define $\lift{Q}$ to be the quiver with vertex set given by
\[\lift{Q}_0=Q_0\sqcup Q_0^+\sqcup Q_0^-\]
where $Q_0^+=\set{i^+:i\in Q_0}$ is a set of formal symbols in bijection with $Q_0$, and similarly for $Q_0^-=\set{i^-:i\in Q_0}$. The set of arrows is given by
\[\lift{Q}_1=Q_1\sqcup\set{\alpha_i:i\in Q_0}\sqcup\set{\beta_i:i\in Q_0}\sqcup\set{\delta_i:i\in Q_0}\sqcup\set{\delta_a:a\in Q_1},\]
and the head and tail functions $\head{},\tail{}\colon\lift{Q}_1\to\lift{Q}_0$ are extended from those of $Q$ by defining
\begin{align*}
\head{\alpha_i}&=i^+,&\head{\beta_i}&=i,&\head{\delta_i}&=i^-,&\head{\delta_a}&=(\tail{a})^-,\\
\tail{\alpha_i}&=i,&\tail{\beta_i}&=i^-,&\tail{\delta_i}&=i^+,&\tail{\delta_a}&=(\head{a})^+.
\end{align*}
The frozen subquiver $\lift{F}$ is defined by
\[
\lift{F}_0=Q_0^+\sqcup Q_0^-,\qquad
\lift{F}_1=\set{\delta_i:i\in Q_0}\sqcup\set{\delta_a:a\in Q_1}.
\]
Note that the head and tail of any arrow in $\lift{F}_1$ lies in $\lift{F}_0$, so these subsets describe a valid subquiver of $\lift{Q}$, which is even full. The quiver $\lift{F}$ is also bipartite, meaning that every vertex is either a source or a sink, and so it has no paths of length greater than $1$; precisely, the vertices in $Q_0^+$ are sources in $\lift{F}$, and those in $Q_0^-$ are sinks. As a vertex of $\lift{Q}$, each $i^+$ has unique incoming arrow $\alpha_i$ and each $i^-$ has unique outgoing arrow $\beta_i$.

Finally, we define a potential $\lift{W}$ on $\lift{Q}$ by
\[\lift{W}=W+\sum_{i\in Q_0}\beta_i\delta_i\alpha_i-\sum_{a\in Q_1}a\beta_{\tail{a}}\delta_a\alpha_{\head{a}},\]
and let
\[A_{Q,W}=\frjac{\lift{Q}}{\lift{F}}{\lift{W}}\]
be the frozen Jacobian algebra determined by $(\lift{Q},\lift{F},\lift{W})$. We denote the boundary algebra of $A_{Q,W}$ by $B_{Q,W}=eA_{Q,W}e$, where $e=\sum_{i\in Q_0}(e_i^++e_i^-)$ is the frozen idempotent of $A_{Q,W}$.
\end{defn}

\begin{nb}
To aid legibility, if the vertices $i^+$ or $i^-$ appear as subscripts, we will usually move the sign into a superscript, as in the above expression for the frozen idempotent. When $W=0$, we will typically drop it from the notation; for example, we write $A_Q=A_{Q,0}$. The reader is warned that $\lift{0}$ is not the zero potential on $\lift{Q}$.
\end{nb}

Since $\lift{W}$ has a straightforward description in terms of $W$, so do the defining relations of $A_{Q,W}$; these form the set $R$ consisting of
\begin{equation}
\label{eq:relations}
\der{a}{\lift{W}}=\der{a}{W}-\beta_{\tail{a}}\delta_a\alpha_{\head{a}},\quad
\der{\alpha_i}{\lift{W}}=\beta_i\delta_i-\sum_{\substack{\gamma\in Q_1\\\head{\gamma}=i}}\gamma\beta_{\tail{\gamma}}\delta_\gamma,\quad
\der{\beta_i}{\lift{W}}=\delta_i\alpha_i-\sum_{\substack{\gamma\in Q_1\\\tail{\gamma}=i}}\delta_\gamma\alpha_{\head{\gamma}}\gamma,
\end{equation}
for $a\in Q_1$ and $i\in Q_0$. Having such an explicit generating set of relations will prove to be extremely useful later in the paper.

To be able to apply Theorem~\ref{t:icytofrobcat}, we wish to show that $A_{Q,W}$ is bimodule internally $3$-Calabi--Yau with respect to its frozen idempotent $e$, in the sense of Definition~\ref{d:bi3cy}. We will do this, under a mild assumption on $(Q,W)$, in Section~\ref{s:frjacCY}, but first we give some examples.

\begin{eg}
\label{e:a2}
The quiver with potential $(Q,0)$, for $Q$ an $\type{A}_2$ quiver, provides the most basic example revealing all of the combinatorial features of the construction. In this case, we have
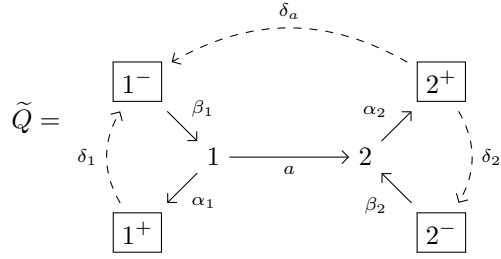
\[\lift{Q}=\mathord{\begin{tikzpicture}[baseline={(current bounding box.center)},yscale=-1]
\node at (-1,0) (1) {$1$};
\node at (1,0) (2) {$2$};
\node at (-2,1) (1+) {$\boxed{1^+}$};
\node at (-2,-1) (1-) {$\boxed{1^-}$};
\node at (2,1) (2-) {$\boxed{2^-}$};
\node at (2,-1) (2+) {$\boxed{2^+}$};
\path[-angle 90,font=\scriptsize]
	(1) edge node[below] {$a$} (2)
	(1) edge node[below right] {$\alpha_1$} (1+)
	(1-) edge node[above right] {$\beta_1$} (1)
	(2) edge node[above left] {$\alpha_2$} (2+)
	(2-) edge node[below left] {$\beta_2$} (2);
\path[\frozen,-angle 90,font=\scriptsize]
	(1+) edge[bend right] node[left] {$\delta_1$} (1-)
	(2+) edge[bend right] node[right] {$\delta_2$} (2-)
	(2+) edge[bend left] node[above] {$\delta_a$} (1-);
\end{tikzpicture}}\]
with $\lift{F}$ indicated by the boxed vertices and dashed arrows. The potential on this ice quiver is
\[\lift{W}=\beta_1\delta_1\alpha_1+\beta_2\delta_2\alpha_2-a\beta_1\delta_a\alpha_2.\]
One can check that the frozen Jacobian algebra $A_Q$ attached to this data is isomorphic to the endomorphism algebra of a cluster-tilting object in the Frobenius cluster category consisting of those modules for the preprojective algebra of type $\type{A}_4$ with socle supported at a fixed bivalent vertex \cite{GLS-PFVs}. For an explanation of why the categories arising in \cite{GLS-PFVs} are Frobenius cluster categories, see \cite[Eg.~3.11]{Pressland-iCY}.
\end{eg}
\begin{eg}
\label{e:3-cycle}
Now let $(Q,W)$ be the quiver with potential in which
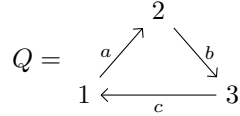
\[Q=\mathord{\begin{tikzpicture}[baseline=0,xscale=1.5]
\node at (-30:0.75) (3) {$3$};
\node at (-150:0.75) (1) {$1$};
\node at (90:0.75) (2) {$2$};
\path[-angle 90,font=\scriptsize]
(3) edge node[below] {$c$} (1)
(1) edge node[left] {$a$} (2)
(2) edge node[right] {$b$} (3);
\end{tikzpicture}}\]
and $W=cba$. The Jacobian algebra is a cluster-tilted algebra of type $\type{A}_3$, and has infinite global dimension (like any non-hereditary cluster-tilted algebra \cite[Cor.~2.1]{kellerclustertilted}). In this case
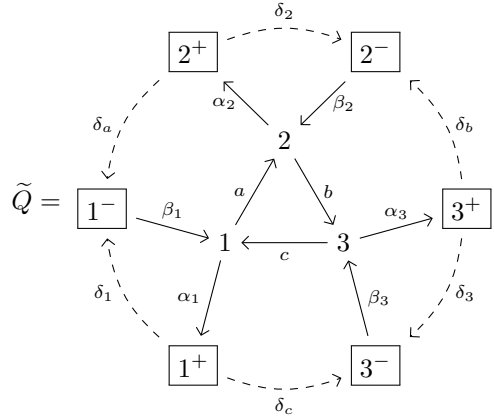
\[\widetilde{Q}=\mathord{\begin{tikzpicture}[baseline={(current bounding box.center)},scale=1.2]
\node at (-30:0.75) (3) {$3$};
\node at (-150:0.75) (1) {$1$};
\node at (90:0.75) (2) {$2$};
\node at (60:2) (2-) {$\boxed{2^-}$};
\node at (0:2) (3+) {$\boxed{3^+}$};
\node at (-60:2) (3-) {$\boxed{3^-}$};
\node at (-120:2) (1+) {$\boxed{1^+}$};
\node at (180:2) (1-) {$\boxed{1^-}$};
\node at (120:2) (2+) {$\boxed{2^+}$};
\path[-angle 90,font=\scriptsize]
(3) edge node[below] {$c$} (1)
(1) edge node[left] {$a$} (2)
(2) edge node[right] {$b$} (3)
(2-) edge node[right] {$\beta_2$} (2)
(3) edge node[above] {$\alpha_3$} (3+)
(3-) edge node[right] {$\beta_3$} (3)
(1) edge node[left] {$\alpha_1$} (1+)
(1-) edge node[above] {$\beta_1$} (1)
(2) edge node [left] {$\alpha_2$} (2+);
\path[-angle 90,\frozen,font=\scriptsize]
(1+) edge [bend right=20pt] node[below] {$\delta_c$} (3-)
(3+) edge [bend left=20pt] node[right] {$\delta_3$} (3-)
(3+) edge [bend right=20pt] node[right] {$\delta_b$} (2-)
(2+) edge [bend left=20pt] node[above] {$\delta_2$} (2-)
(2+) edge [bend right=20pt] node[left] {$\delta_a$} (1-)
(1+) edge [bend left=20pt] node[left] {$\delta_1$} (1-);
\end{tikzpicture}}\]
with $\lift{F}$ again indicated by boxed vertices and dashed arrows. The potential is
\[\lift{W}=cba+\beta_1\delta_1\alpha_1+\beta_2\delta_2\alpha_2+\beta_3\delta_3\alpha_3-a\beta_1\delta_a\alpha_2-b\beta_2\delta_b\alpha_3-c\beta_3\delta_c\alpha_1.\]
The associated frozen Jacobian algebra $A_{Q,W}$ also arises from a dimer model on a disk with six marked points on its boundary \cite{BKM}, and is isomorphic to the endomorphism algebra of a cluster-tilting object in Jensen--King--Su's categorification of the cluster algebra structure on the Grassmannian $\Grass{2}{6}$ \cite{JKS}. This category is again a Frobenius cluster category \cite[Eg.~3.12]{Pressland-iCY}. Unlike the first example, this algebra is infinite-dimensional. However, it is Noetherian, so Theorem~\ref{t:icytofrobcat} still applies.
\end{eg}

Since in Examples~\ref{e:a2} and \ref{e:3-cycle} the algebra $A_{Q,W}$ is the endomorphism algebra of a cluster-tilting object in a Frobenius cluster category, it is internally $3$-Calabi--Yau with respect to its frozen idempotent---that is, it has the properties in Remark~\ref{r:i3cy}---by a result of Keller--Reiten \cite[\S5.4]{kellerclustertilted} (see also \cite[Cor.~3.10]{Pressland-iCY}). This foreshadows Theorem~\ref{t:bi3cy} below, which states that under mild assumptions on $(Q,W)$ the algebra $A_{Q,W}$ has the (a priori stronger) bimodule internal Calabi--Yau property from Definition~\ref{d:bi3cy}.

\section{Calabi--Yau properties for frozen Jacobian algebras}
\label{s:frjacCY}

In this section, we will recall from \cite[\S5]{Pressland-iCY} a bimodule complex $0\to\res{A}\to A\to0$ for a frozen Jacobian algebra $A$, exactness of which implies that $A$ is bimodule internally $3$-Calabi--Yau with respect to its frozen idempotent. We show that this complex is exact when $A=A_{Q,W}$ is as in Definition~\ref{d:bigQP}. Thus these frozen Jacobian algebras satisfy the most complicated condition in Theorem~\ref{t:icytofrobcat}. This constitutes the main step in the proof of Theorem~\ref{t:main-cat-thm}.

Our arguments exploit the fact that the algebras we consider are pseudocompact, and so we begin with some useful generalities on such algebras. Further background can be found in \cite{gabrielcategories,brumerpseudocompact,simsoncoalgebras,vandenberghcalabiyau,iusenkopath}, for example.

\begin{defn}
\label{d:pseudocompact}
Let $A$ be a topological $\KK$-algebra. We say that $A$ is \emph{pseudocompact} if there is a system $\mathcal{I}=\{I\}$ of neighbourhoods of $0\in A$ consisting of (open) ideals such that $A/I$ is finite-dimensional for all $I\in\mathcal{I}$ and the natural map
\[A\to{\invlim}_IA/I\]
is an isomorphism. The \emph{Jacobson radical} of $A$, denoted by $J$, is the intersection of maximal closed left ideals of $A$, and we say that $A$ is \emph{$J$-adically complete} if the natural map $A\to\invlim_nA/J^n$ is an isomorphism.
\end{defn}

\begin{prop}
\label{p:cpa-pseudocompact}
If $Q$ is a finite quiver, the complete path algebra $\cpa{\KK}{Q}$ is pseudocompact and $J$-adically complete, where $J$ is the Jacobson radical.
\end{prop}
\begin{proof}
The algebra $\cpa{\KK}{Q}$ may be identified with the continuous dual of the path coalgebra of $Q^\op$, i.e.\ the space of continuous functionals on this coalgebra; see for example \cite[Prop.~8.1(c)]{simsoncoalgebras}. The continuous dual of any coalgebra is pseudocompact \cite{simsoncoalgebrastame}.

The Jacobson radical of $\cpa{\KK}{Q}$ is generated by the arrows. In particular, $\cpa{\KK}{Q}/J^2$ is finite-dimensional, since $Q$ is a finite quiver, and so $\cpa{\KK}{Q}$ is $J$-adically complete by \cite[Prop.~2.7]{iusenkopath}.
\end{proof}

\begin{prop}
\label{p:pseudocompact-quot}
Let $Q$ be a finite quiver. Then for any ideal $I\leq\cpa{\KK}{Q}$ which is closed in the $J$-adic topology, for $J$ the Jacobson radical, the quotient algebra $A=\cpa{\KK}{Q}/I$ is pseudocompact.
\end{prop}
\begin{proof}
This statement (in a more general form) can be found in The Stacks Project \S00M9\footnote{See \href{https://stacks.math.columbia.edu/tag/00M9}{\url{https://stacks.math.columbia.edu/tag/00M9}}}, but we give a direct argument. For each $n\in\ZZ$, there is a short exact sequence
\begin{equation}
\label{eq:ideal-seq}
\begin{tikzcd}
0\arrow{r}&(I+J^n)/J^n\arrow{r}&\cpa{\KK}{Q}/J^n\arrow{r}&A/\pi J^n\arrow{r}&0,
\end{tikzcd}
\end{equation}
where $\pi\colon\cpa{\KK}{Q}\to A=\cpa{\KK}{Q}/I$ is the projection. In particular, it follows that $A/\pi J^n$ is finite-dimensional. Note that the ideals $\pi J^n$ form a system of open neighbourhoods of $0$ in $A$, since $\pi^{-1}(\pi J^n)=I+J^n$ is open in $\cpa{\KK}{Q}$.

The sequence \eqref{eq:ideal-seq} is the $n$-th component of a short exact sequence of inverse systems, each of which has only surjective morphisms and so satisfies (in a trivial way) the Mittag-Leffler condition. Thus we obtain a short exact sequence of inverse limits
\[\begin{tikzcd}
0\arrow{r}&\invlim_n(I+J^n)/J^n\arrow{r}&\invlim_n\cpa{\KK}{Q}/J^n\arrow{r}&\invlim_nA/\pi J^n\arrow{r}&0.
\end{tikzcd}\]
The left-hand term of this sequence may be naturally identified with $\bigcap_{n\geq1}(I+J^n)$, which is the closure of $I$ in the $J$-adic topology, and is thus equal to $I$ since $I$ is closed. Moreover, the middle term is $\cpa{\KK}{Q}$ since this algebra is $J$-adically complete by Proposition~\ref{p:cpa-pseudocompact}. Hence the right-hand term is $A$, which is therefore pseudocompact.
\end{proof}

\begin{cor}
\label{c:frjac-pseudocompact}
For any ice quiver with potential $(Q,F,W)$, the frozen Jacobian algebra $\frjac{Q}{F}{W}$ is pseudocompact.
\end{cor}
\begin{proof}
By definition, $\frjac{Q}{F}{W}$ is the quotient of $\cpa{\KK}{Q}$ by a closed ideal. Thus the result follows by Proposition~\ref{p:pseudocompact-quot}.
\end{proof}

We now return to the setting of ice quivers with potential. Let $(Q,F,W)$ be an ice quiver with potential and write $A=\frjac{Q}{F}{W}$. Recall that $Q_0^\mut=Q_0\setminus F_0$ and $Q_1^\mut=Q_1\setminus F_1$ denote the sets of mutable vertices and unfrozen arrows of $Q$ respectively. For $i\in Q_0$, we write $\Head{i}$ for the set of arrows of $Q$ with head $i$, and $\Tail{i}$ for the set of arrows of $Q$ with tail $i$. We denote by $\mHead{i}$ and $\mTail{i}$ the intersection of $Q_1^\mut$ with $\Head{i}$ and $\Tail{i}$ respectively. Recall from Remark~\ref{r:completion} that the quotient $S:=A/\Jacrad{A}$ of $A$ by its Jacobson radical is a semi-simple algebra, isomorphic as a left $A$-module to the direct sum of the vertex simple $A$-modules. (This property is sometimes \cite[\S3.2]{iusenkopath} called \emph{pointedness}.) Thus $S$ has a basis given by the vertex idempotents $e_i$ for $i\in Q_0$. For the remainder of this section, we write $\tensor=\tensor_S$.

Introduce formal symbols $\rho_\alpha$ for each $\alpha\in Q_1$ and $\omega_v$ for each $v\in Q_0$, and define $S$-bimodule structures on the vector spaces
\begin{align*}
V_0&=\bigdsum_{i\in Q_0}\KK\idemp{i},&
V_1&=\bigdsum_{\alpha\in Q_1}\KK\alpha,&
V_2&=\bigdsum_{\alpha\in Q_1^\mut}\KK\rel{\alpha},&
V_3&=\bigdsum_{i\in Q_0^\mut}\KK\omega_i,
\end{align*}
via the formulae
\[
\idemp{i}\cdot\idemp{i}\cdot\idemp{i}=\idemp{i},\qquad
\idemp{\head{\alpha}}\cdot\alpha\cdot\idemp{\tail{\alpha}}=\alpha,\qquad
\idemp{\tail{\alpha}}\cdot\rel{\alpha}\cdot\idemp{\head{\alpha}}=\rel{\alpha},\qquad
\idemp{i}\cdot\omega_i\cdot\idemp{i}=\omega_i.
\]
Next we describe various morphisms $\mu_k$ of $A$-bimodules. Since $V_0\cong S$, the $A$-bimodule $A\tensor V_0\tensor A$ is canonically isomorphic to $A\tensor A$. Composing this isomorphism with the multiplication map for $A$, we obtain a map $\mu_0\colon A\tensor V_0\tensor A\to A$. Define $\mu_1\colon A\tensor V_1\tensor A\to A\tensor V_0\tensor A$ by
\[\mu_1(x\tensor\alpha\tensor y)=x\tensor\idemp{\head{\alpha}}\tensor\alpha y-x\alpha\tensor\idemp{\tail{\alpha}}\tensor y.\]
For any path $p=\alpha_m\dotsm\alpha_1$ of $\KK Q$, we may define
\[\Delta_\alpha(p)=\sum_{\alpha_i=\alpha}\alpha_m\dotsm\alpha_{i+1}\tensor\alpha_i\tensor\alpha_{i-1}\dotsm\alpha_1,\]
and extend by linearity and continuity to obtain a map $\Delta_\alpha\colon\cpa{\KK}{Q}\to A\tensor V_1\tensor A$. We then define $\mu_2\colon A\tensor V_2\tensor A\to A\tensor V_1\tensor A$ by
\[\mu_2(x\tensor\rel{\alpha}\tensor y)=\sum_{\beta\in Q_1}x\Delta_\beta(\der{\alpha}{W})y.\]
Finally, define $\mu_3\colon A\tensor V_3\tensor A\to A\tensor V_2\tensor A$ by
\[\mu_3(x\tensor\omega_v\tensor y)=\sum_{\alpha\in\Tail{v}}x\tensor\rel{\alpha}\tensor\alpha y-\sum_{\beta\in\Head{v}}x\beta\tensor\rel{\beta}\tensor y.\]
Note that we have $A\otimes V_k\otimes A\otimes_AS=A\otimes V_k\otimes S=A\otimes V_k$ for $0\leq k\leq 3$. Thus for each $1\leq k\leq 3$ we obtain a map $\bar{\mu}_k:=(\mu_k\otimes_AS)\colon A\otimes V_k\to A\otimes V_{k-1}$.

\begin{lem}
\label{l:diff-shape}
Let $1\leq k\leq 3$ and consider $x\otimes v\otimes y\in A\otimes V_k\otimes A$, where $x,y\in A$ and $v\in V_k$. Then
\[\mu_k(x\otimes v\otimes y)=\bar{\mu}_k(x\otimes v)\otimes y + r,\]
where $r\in A\otimes V_{k-1}\otimes Jy$.
\end{lem}
\begin{proof}
This follows directly from the definition of the maps $\mu_k$, recalling that the Jacobson radical $J$ is generated by the arrows of $Q$, and that $A/J=S$.
\end{proof}

\begin{defn}
\label{d:frjacres}
For an ice quiver with potential $(Q,F,W)$ with frozen Jacobian algebra $A$, denote by $\res{A}$ the complex of $A$-bimodules with non-zero terms
\[\begin{tikzcd}
A\tensor V_3\tensor A\arrow{r}{\mu_3}&A\tensor V_2\tensor A\arrow{r}{\mu_2}&A\tensor V_1\tensor A\arrow{r}{\mu_1}&A\tensor V_0\tensor A
\end{tikzcd}\]
and $A\tensor\KK Q_0\tensor A$ in degree $0$. That this sequence is indeed a complex is \cite[Lem.~5.5]{Pressland-iCY}.
\end{defn}

By \cite[Thm.~5.7]{Pressland-iCY}, if $A$ is a frozen Jacobian algebra such that
\begin{equation}
\label{eq:bimod-res}
\begin{tikzcd}
0\arrow{r}&\res{A}\arrow{r}{\mu_0}&A\arrow{r}&0
\end{tikzcd}
\end{equation}
is exact, then $A$ is bimodule internally $3$-Calabi--Yau with respect to the idempotent $e=\sum_{i\in F_0}\idemp{i}$. Our goal for the remainder of the section is to check exactness of \eqref{eq:bimod-res} when $A=A_{Q,W}$ is the algebra from Definition~\ref{d:bigQP}. The following lemma allows us to simplify this calculation; its proof is based on \cite[Prop.~7.5]{Broomhead}, but adapted to exploit pseudocompactness of $A$ in place of a suitable grading.

\begin{lem}
\label{l:simplify-exactness}
The complex \eqref{eq:bimod-res} is exact if and only if
\begin{equation}
\label{eq:proj-res-simples}
\begin{tikzcd}
0\arrow{r}&\res{A}\tensor_AS\arrow{r}{\bar{\mu}_0}&S\arrow{r}&0
\end{tikzcd}
\end{equation}
is exact.
\end{lem}
\begin{proof}
The forward implication holds since $\res{A}\map{\mu_0}A$ is perfect as a complex of right $A$-modules, and so remains exact under $-\tensor_AM$ for any $M\in\Modcat{A}$.

For the converse implication, we use pseudocompactness. Abbreviate $\bar{\mu}_{k,n}=\mu_k\otimes_AA/J^n$ (so that $\bar{\mu}_{k,1}=\bar{\mu}_k$ as defined above). We will show inductively that
\begin{equation}
\label{eq:proj-res-n}
\begin{tikzcd}
0\arrow{r}&\res{A}\tensor_AA/J^n\arrow{r}{\bar{\mu}_{0,n}}&A/J^n\arrow{r}&0
\end{tikzcd}
\end{equation}
is exact for all $n$; the assumed exactness of \eqref{eq:proj-res-simples} gives the base case $n=1$ of this induction. Then the result will follow as in the proof of Proposition~\ref{p:pseudocompact-quot} from the fact that \eqref{eq:proj-res-n} remains exact under taking inverse limits, noting again that the relevant inverse systems satisfy the Mittag-Leffler condition since all their morphisms are surjective.

That \eqref{eq:proj-res-n} is exact at $A/J^n$ (and indeed, that \eqref{eq:bimod-res} is exact at $A$) follows since $\mu_0$ is multiplication in the unital algebra $A$, so we concentrate on the other terms, each of which has the form $A\otimes V_k\otimes A/J^n$ for some $S$-bimodule $V_k$, adopting the convention that $V_k=0$ for $k>3$.

Choose $k\geq 0$, let $\phi_0\in A\otimes V_k\otimes A/J^n$ for some $n\geq2$, and write $\phi_0'$ for the projection of $\phi_0$ to $A\otimes V_k\otimes A/J^{n-1}$. If $\phi_0$ is closed, i.e.\ if $\bar{\mu}_{k,n}(\phi_0)=0$, then $\phi'_0$ is also closed, and and so by induction there is $\psi_0'\in A\otimes V_{k+1}\otimes A/J^{n-1}$ with $\bar{\mu}_{k+1,n-1}(\psi')=\phi'$. Choose any $\psi_0\in A\otimes V_{k+1}\otimes A/J^n$ projecting to $\psi'_0$. Then $\phi_1:=\phi_0-\bar{\mu}_{k+1,n}(\psi_0)$ projects to $\phi_0'-\bar{\mu}_{k,n-1}(\psi_0')=0$, and hence $\phi_1\in A\otimes V_k\otimes J^{n-1}/J^n$. Moreover, $\phi_1$ is closed.

Now let $Y$ be a basis of the finite-dimensional vector space $J^{n-1}/J^n$, so that we may write
\[\phi_1=\sum_{y\in Y}u_y\otimes y\]
for some $u_y\in A\otimes V_k$. Since each $y\in J^{n-1}/J^n$ and $\phi_1$ is closed, we have
\[0=\sum_{y\in Y}\bar{\mu}_{k,n}(u_y\otimes y)=\sum_{y\in Y}\bar{\mu}_k(u_y)\otimes y\]
by Lemma~\ref{l:diff-shape}. It then follows by linear independence of the set $Y$ that $\bar{\mu}_k(u_y)=0$ for all $y\in Y$. By the assumed exactness of \eqref{eq:proj-res-simples}, there are therefore $v_y\in A\otimes V_{k+1}$ with $\bar{\mu}_{k+1}(v_y)=u_y$.

Now let
\[\psi_1=\sum_{y\in Y}v_y\otimes y\in A\otimes V_{k+1}\otimes A/J^n.\]
Using Lemma~\ref{l:diff-shape} again, we may calculate
\[\bar{\mu}_{k+1,n}(\psi_1)=\sum_{y\in Y}\bar{\mu}_{k+1}(v_y)\otimes y=\sum_{y\in Y}u_y\otimes y=\phi_1,\]
and so $\phi_0=\bar{\mu}_{k+1,n}(\psi_0+\psi_1)$ is exact. Thus each sequence \eqref{eq:proj-res-n}, and hence the limit \eqref{eq:bimod-res}, is exact as required.
\end{proof}

Thus it suffices for us to show that \eqref{eq:proj-res-simples} is exact. This complex decomposes along with $S$, so that its exactness is equivalent to the exactness of
\begin{equation}
\label{eq:proj-res-simp-i}
\begin{tikzcd}
0\arrow{r}&\res{A}\tensor_A\simp{i}\arrow{r}{\mu_0\tensor_AS_i}&\simp{i}\arrow{r}&0
\end{tikzcd}
\end{equation}
for each $i\in \lift{Q}_0$, where $\simp{i}$ denotes the vertex simple left $A$-module at $i$.
We may calculate the terms in this complex as
\begin{align}
\label{eq:resA-isos}
\begin{split}
A\tensor\KK Q_1\tensor A\tensor_A\simp{i}&\iso\bigdsum_{b\in\Tail{i}}Ae_{\head{b}},\\
A\tensor\KK Q_2^\mut\tensor A\tensor_A\simp{i}&\iso\bigdsum_{a\in\mHead{i}}Ae_{\tail{a}},\\
A\tensor\KK Q_3^\mut\tensor A\tensor_A\simp{i}&\iso{\begin{cases}A\idemp{i},&i\in Q_0^\mut,\\0,&i\in F_0.\end{cases}}
\end{split}
\end{align}
In the first two cases, the right-hand sides are of the form $\bigoplus_{a\in X}A\idemp{v(a)}$, where $X$ is a set of arrows, and $v\colon X\to Q_0$. We will denote by $x\mapsto x\otimes[a]$ the inclusion $A\idemp{v(a)}\to\bigoplus_{a\in X}A\idemp{v(a)}$ mapping the domain to the summand indexed by $a$; this helps us to distinguish these various inclusions when $v$ is not injective. As a consequence, a general element of the direct sum is $x=\sum_{a\in X}x_a\otimes[a]$ for $x_a\in A\idemp{v(a)}$.

Each arrow $b\in Q_1$ has an associated \emph{right derivative} $\rightder{b}\colon\cpa{\KK}{Q}\to\cpa{\KK}{Q}$, which is linear and continuous and given by
\begin{equation}
\label{eq:rightder}
\rightder{b}(\alpha_\ell\dotsm\alpha_1)=\begin{cases}\alpha_\ell\dotsm\alpha_2,&\alpha_1=b,\\0,&\alpha_1\ne b\end{cases}
\end{equation}
on paths. Using this map, we may write
\[(\mu_2\tensor_A\simp{i})(x)=\sum_{b\in\Tail{i}}\bigg(\sum_{a\in\mHead{i}}x_a\rightder{b}{\der{a}{W}}\bigg)\tensor[b].\]
Calculating $\mu_k\otimes_A S_i$ for the other values of $k$ is similar, and more straightforward.

\begin{rem}
\label{r:BK}
By standard results on presentations of algebras, see for example Butler--King \cite{butlerminimal}, the complex \eqref{eq:bimod-res} is always exact in degrees $-1$, $0$ and $1$, and hence the same is true for each complex \eqref{eq:proj-res-simp-i}. Thus in the following argument, we need only check exactness of \eqref{eq:proj-res-simp-i} at $A\otimes V_3\otimes S_i$ and $A\otimes V_2\otimes S_i$.
\end{rem}

Now let $(Q,W)$ be a quiver with potential, and let $A=A_{Q,W}=\frjac{\lift{Q}}{\lift{F}}{\lift{W}}$.
The following series of results will establish exactness of the complex \eqref{eq:proj-res-simples} for $A$ for each vertex $i\in\lift{Q}_0$, using heavily the explicit set $R$ of defining relations for $A$ given in \eqref{eq:relations}.

\begin{lem}
\label{l:vertex-test}
Let $i\in\lift{Q}_0^\mut=Q_0$. If $y\in Ae_i$ satisfies $y\beta_i=0$, then $y=0$.
\end{lem}
\begin{proof}
Let $\lift{y}$ be an arbitrary lift of $y$ to $\cpa{\KK}{\lift{Q}}e_i$. Now assume $y\beta_i=0$, so $\lift{y}\beta_i\in\close{\Span{R}}$. Since every term of $\lift{y}\beta_i$, when written in the basis of paths of $\lift{Q}$, ends with the arrow $\beta_i$, but no term of any element of $R$ has a term ending with $\beta_i$, we must be able to write
$\lift{y}\beta_i=z\beta_i$
for some $z\in\close{\Span{R}}e_i$. Comparing terms, we see that $\lift{y}=z\in\close{\Span{R}}$, and so $y=0$ in $A$.
\end{proof}

\begin{prop}
\label{p:mut-exact-3}
For $i\in\lift{Q}_0^\mut$, the map $\mu_3\tensor_A\simp{i}\colon A\idemp{i}\to\displaystyle\bigdsum_{a\in\mHead{i}} Ae_{\tail{a}}$ is injective.
\end{prop}
\begin{proof}
We have $(\mu_3\tensor_A\simp{i})(y)=\sum_{a\in\mHead{i}}-ya\otimes[a]$. Since $i$ is mutable, $\mHead{i}=\Head{i}$ contains $\beta_i$ by the construction of $(\lift{Q},\lift{F})$. Thus if $(\mu_3\tensor_A\simp{i})(y)=0$, it follows that $-y\beta_i\tensor[\beta_i]=0$, hence $y\beta_i=0$, and so $y=0$ by Lemma~\ref{l:vertex-test}.
\end{proof}

\begin{lem}
\label{l:arrow-test}
Let $i\in\lift{Q}_0^\mut$. For each $a\in\mHead{i}$, pick $x_a\in Ae_{\tail a}$. If
$x_{\beta_i}\delta_i=\sum_{a\in\mHead{i}\setminus\set{\beta_i}}x_a\beta_{\tail{a}}\delta_a$,
then there exists $y\in Ae_i$ such that $x_a=ya$ for each $a\in\mHead{i}$.
\end{lem}

\begin{proof}
Pick a lift $\lift{x}_a\in\cpa{\KK}{\lift{Q}}$ of each $x_a$. Writing
$p=\lift{x}_{\beta_i}\delta_i-\sum_{a\in\mHead{i}\setminus\set{\beta_i}}\lift{x}_a\beta_{\tail{a}}\delta_a$,
our assumption on the $x_a$ is equivalent to $p\in\close{\Span{R}}$. Since every term of $p$ ends with either $\delta_i$ or $\beta_{\tail{a}}\delta_a$ for some $a\in\mHead{i}\setminus\set{\beta_i}$, and the only element of $R$ including terms ending with these arrows is $\beta_i\delta_i-\sum_{a\in\mHead{i}\setminus\set{\beta_i}}a\beta_{\tail{a}}\delta_a$, we can write
\[p=z_i\delta_i+\sum_{a\in\mHead{i}\setminus\set{\beta_i}}z_a\beta_{\tail{a}}\delta_a+y\Big(\beta_i\delta_i-\sum_{a\in\mHead{i}\setminus\set{\beta_i}}a\beta_{\tail{a}}\delta_a\Big),\]
where $z_i\in\close{\Span{R}}e_i^-$, $z_a\in\close{\Span{R}}e_{\tail{a}}$ and $y\in\cpa{\KK}{\lift{Q}}e_i$. Comparing terms in our two expressions for $p$, we see that $\lift{x}_{\beta_i}=y\beta_i+z_i$ and $\lift{x}_{a}=ya-z_a$.
Since $z_i,z_a\in\close{\Span{R}}$, when we pass to the quotient algebra $A$ we see that $x_{\beta_i}=y\beta_i$ and $x_a=ya$, as required.
\end{proof}

\begin{prop}
\label{p:mut-exact-2}
For $i\in\lift{Q}_0^\mut$, we have $\ker(\mu_2\tensor_A\simp{i})=\im(\mu_3\tensor_A\simp{i})$.
\end{prop}

\begin{proof}
Since $\res{A}$ is a complex, it is enough to show that
$\ker(\mu_2\tensor_A\simp{i})\subseteq\im(\mu_3\tensor_A\simp{i})$.
Let $x=\sum_{a\in\mHead{i}}x_a\tensor[a]\in\bigdsum_{a\in\mHead{i}}Ae_{\tail{a}}$. Then
\[(\mu_2\tensor_A\simp{i})(x)=\sum_{b\in\Tail{i}}\bigg(\sum_{a\in\mHead{i}}x_a\rightder{b}{\der{a}{\lift{W}}}\bigg)\tensor[b]\in\bigdsum_{b\in\Tail{i}}Ae_{\head{b}}.\]
In particular, $\alpha_i\in\Tail{i}$, so if $x\in\ker(\mu_2\otimes_A\simp{i})$, we have
$\left(\sum_{a\in\mHead{i}}x_a\rightder{\alpha_i}{\der{a}{\lift{W}}}\right)\otimes[\alpha_i]=0$.
Using the explicit expressions for the relations $\der{a}{\lift{W}}$ computed in \eqref{eq:relations}, we see that
\[\sum_{a\in\mHead{i}}x_a\rightder{\alpha_i}\der{a}{\lift{W}}=x_{\beta_i}\delta_i-\sum_{a\in\mHead{i}\setminus\set{\beta_i}}x_a\beta_{\tail{a}}\delta_a=0,\]
and so by Lemma~\ref{l:arrow-test} there exists $y\in A\idemp{i}$ such that $x_a=ya$ for each $a$. It follows that $x=(\mu_3\tensor_A\simp{i})(y)$, as required.
\end{proof}

\begin{prop}
\label{p:frozen-res}
If $i^{\pm}\in\lift{F}_0$ then the complex $\begin{tikzcd}[column sep=32pt]0\arrow{r}&\res{A}\tensor_A\simp{i}^\pm\arrow{r}{\mu_0\tensor_AS_i^\pm}&\simp{i}^\pm\arrow{r}&0\end{tikzcd}$ is exact.
\end{prop}
\begin{proof}
Since $i^{\pm}\in\lift{F}_0$, our complex is zero in degree $-3$. By Remark~\ref{r:BK}, we need only to check that $\mu_2\tensor_A\simp{i}^{\pm}$ is injective. Since $\mHead{i^-}=\varnothing$, the complex $\res{A}\tensor_A\simp{i}^-$ is also zero in degree $-2$, so we are left to consider $\mu_2\tensor_A\simp{i}^+$.

Since $\mHead{i^+}=\set{\alpha_i}$, we have $\mu_2\tensor_A\simp{i}^+\colon Ae_i\to\bigdsum_{b\in\Tail{i}^+}Ae_{\head{b}}$. Let $x\in Ae_i$, for which we calculate
\[(\mu_2\tensor_A\simp{i}^+)(x)=\sum_{b\in\Tail{i}^+}x\rightder{b}{\der{\alpha_i}{\lift{W}}}\otimes[b].\]
Assume this is $0$. Considering $\delta_i\in\Tail{i}^+$, we see that we must have $0=x\rightder{\delta_i}{\der{\alpha_i}{\lift{W}}}=x\beta_i$. 
By Lemma~\ref{l:vertex-test}, it follows that $x=0$, and $\mu_2\tensor_A\simp{i}^+$ is injective as required.
\end{proof}

Summarising the discussion above, we are able to establish the desired internal Calabi--Yau property for $A_{Q,W}$.

\begin{thm}
\label{t:bi3cy}
For any quiver with potential $(Q,W)$, the algebra $A_{Q,W}=\frjac{\lift{Q}}{\lift{F}}{\lift{W}}$ is bimodule internally $3$-Calabi--Yau with respect to the frozen idempotent $e=\sum_{i\in Q_0}(e_i^++e_i^-)$.
\end{thm}
\begin{proof}
Given Remark~\ref{r:BK}, Propositions~\ref{p:mut-exact-3} and \ref{p:mut-exact-2} show that \eqref{eq:proj-res-simp-i} is exact when $i$ is a mutable vertex, whereas Proposition~\ref{p:frozen-res} deals with the case that $i$ is frozen. Thus \eqref{eq:proj-res-simples} is exact and so, since $A_{Q,W}$ is pseudocompact by Corollary~\ref{c:frjac-pseudocompact}, it follows from Lemma~\ref{l:simplify-exactness} that \eqref{eq:bimod-res} is also exact. Hence we obtain the result by \cite[Thm.~5.7]{Pressland-iCY}.
\end{proof}

Since for $A=A_{Q,W}$ we have $A/AeA=\jac{Q}{W}$ by construction, Theorem~\ref{t:keller-analogue} from the introduction is an immediate consequence. We record a further corollary, obtained by combining Theorem~\ref{t:bi3cy} with Theorem~\ref{t:icytofrobcat}.

\begin{cor}
\label{c:frobcat-constr}
Let $(Q,W)$ be a Jacobi-finite quiver with potential. Assume that $A=A_{Q,W}$ is Noetherian, let $e=\sum_{i\in Q_0}(e_i^++e_i^-)$ be the frozen idempotent of $A$, let $B=eAe$ be its boundary algebra, and let $T=eA$. Then $T$ is a cluster-tilting object of the stably $2$-Calabi--Yau Frobenius category $\GP(B)$, with $\Endalg{B}{T}\cong A$ and $\stabEndalg{B}{T}\cong\jac{Q}{W}$.\qed
\end{cor}

It is this corollary that we will use to define the Frobenius category $\frobcat_Q$ appearing in Theorem~\ref{t:main-cat-thm} for an acyclic quiver $Q$; to be able to do so, it only remains to check that $A_Q$ is Noetherian in this case, and we will show in Section~\ref{s:acyclic} that it is even finite-dimensional. Unfortunately, at the moment we lack methods for showing that $A_{Q,W}$ is Noetherian in more cases (and expect that $A_{Q,W}$ is not Noetherian general), preventing us from extending Theorem~\ref{t:main-cat-thm} to the case of quivers with cycles. We already observed in Example~\ref{e:3-cycle} that $A_{Q,W}$ is Noetherian (but not finite-dimensional) when $(Q,W)$ is the $3$-cycle with its natural potential. This quiver with potential is Jacobi-finite, so Corollary~\ref{c:frobcat-constr} also applies in this case.

\begin{rem}
\label{r:Ginzburg}
Combinatorially, the algebra $A_{Q,W}$ seems to be related the dg-algebra $\Gamma_{Q,W}$ associated to $(Q,W)$ by Ginzburg \cite[\S4.2]{ginzburgcalabiyau}; the loops in cohomological degree $-2$ in $\Gamma_{Q,W}$ are replaced by the cycles $i\to i^+\to i^-\to i$, and the degree $-1$ arrows are replaced by the paths $\head{a}\to\head{a}^+\to\tail{a}^-\to\tail{a}$. Here we use Amiot's sign conventions \cite[Def.~3.1]{Amiot-ClustCat}, which are opposite to Ginzburg's. By a result of Van den Bergh \cite[Thm.~A.12]{kellerdeformed}, the dg-algebra $\Gamma_{Q,W}$ is bimodule $3$-Calabi--Yau, and we expect this to be related to the fact that $A_{Q,W}$ is bimodule internally $3$-Calabi--Yau with respect to the idempotent determined by the vertices not appearing in Ginzburg's construction. We will show in proving Theorem~\ref{t:main-cat-thm}\ref{t:main-cat-thm-cc} that, when $Q$ is acyclic, the two constructions are also related via a triangle equivalence
\begin{equation}
\label{eq:stab-equiv}
\stabGP(B_Q)\simeq\clustcat{Q}=\per{\Gamma_Q}/\bdcat(\Gamma_Q).
\end{equation}

By a result of Yilin Wu \cite[Lem.~5.10]{Wu-IceQuiver}, it follows from exactness of \eqref{eq:bimod-res} that the relative Ginzburg algebra $\Gamma_{\lift{Q},\lift{F},\lift{W}}$ \cite[Def.~4.20]{Wu-IceQuiver} is concentrated in degree $0$, its $0$-th cohomology being the algebra $A_{Q,W}$. There is thus \cite[Prop.~3.19]{Wu-Higgs} a homotopy pushout diagram
\[\begin{tikzcd}
\dgpreproj_{\lift{F}}\arrow{r}{G}\arrow{d}&A_{Q,W}\arrow{d}\\
*\arrow{r}&\Gamma_{Q,W},
\end{tikzcd}\]
taken in the category of $k$-linear dg-categories with Tabuada's model structure \cite{tabuadastructure}, where the functor $G$ is the deformed relative $3$-Calabi--Yau completion, with respect to $\lift{W}$, of the natural map $\KK\lift{F}\to\KK\lift{Q}$. In particular, this makes $\smash{\dgpreproj_{\lift{F}}}$ the $2$-Calabi--Yau completion of $\KK\lift{F}$ (a dg version of the preprojective algebra of $\lift{F}$). Thus we can reconstruct the Ginzburg dg-algebra $\Gamma_{Q,W}$ from $A_{Q,W}$ (plus enough additional data to recover the dg-functor $G$). It is also possible to obtain the triangle equivalence \eqref{eq:stab-equiv} using this style of argument; see Remark~\ref{r:wu2}.

However, the only properties of the ice quiver with potential $(\lift{Q},\lift{F},\lift{W})$ that we use here are that $\Gamma_{\lift{Q},\lift{F},\lift{W}}$ is concentrated in degree $0$, and that removing the frozen vertices recovers $(Q,W)$; typically there are many ice quivers with potential with these two properties. At present we do not know a natural construction of the specific frozen Jacobian algebra $A_{Q,W}$ from $\Gamma_{Q,W}$, despite the combinatorial similarity of their definitions.
\end{rem}

\section{The acyclic case}
\label{s:acyclic}

Let $(Q,W)$ be a Jacobi-finite quiver with potential, so that the algebra $A_{Q,W}$ from Definition~\ref{d:bigQP} satisfies all of the assumptions of Corollary~\ref{c:frobcat-constr} except possibly Noetherianity. Our goal in this section is to show that if $Q$ is an acyclic quiver, then the algebra $A=A_Q=\frjac{\lift{Q}}{\lift{F}}{\lift{0}}$ is even finite-dimensional, so that this result does indeed apply.

\begin{lem}
\label{l:Q0-to-Q0}
Let $Q$ be a finite quiver and let $p$ be a path in $\lift{Q}$ containing at least one arrow not in $Q_1$, and with $\head{p},\tail{p}\in Q_0$. Then $p$ is in the kernel of the projection $\pi\colon\cpa{\KK}{\lift{Q}}\to A=\frjac{\lift{Q}}{\lift{F}}{\lift{0}}$. Moreover, if $q$ is a path of length at least $5$ containing no arrows of $Q_1$, then $q\in\ker{\pi}$.
\end{lem}
\begin{proof}
Let $\gamma$ be the first arrow of $p$ in $\lift{Q}_1\setminus Q_1$. If $\gamma$ has a predecessor in $p$, then this arrow is in $Q_1$, and so $\tail{\gamma}\in Q_0$. On the other hand, if $\gamma$ is the first arrow of $p$, then $\tail{\gamma}=\tail{p}\in Q_0$ by assumption. By the construction of $\lift{Q}$, it follows that $\gamma=\alpha_i$ for some $i\in Q_0$. Since $\head{p}\in Q_0$, the path $p$ cannot terminate with $\alpha_i$, so this arrow has a successor. Looking again at the definition of $\lift{Q}$, the only options are $\delta_a$ for some $a\in Q_1$ with $\head{a}=i$, or $\delta_i$. We break into two cases.

First assume $\alpha_i$ is followed in $p$ by $\delta_a$ for some $a\in Q_0$. This again cannot be the final arrow of $p$, since $\head{p}\in Q_0$. The only arrow leaving $\head{\delta_a}=i_{\tail{a}}^-$ is $\beta_{\tail{a}}$, so this must be the next arrow of $p$. But after projection to $A$ by $\pi$, we have $\beta_{\tail{a}}\delta_a\alpha_{\head{a}}=\der{a}{W}=0$ since $W=0$, so $p\in\ker{\pi}$.

In the second case, $\alpha_i$ is followed by $\delta_i$. As above, $\delta_i$ must be followed in $p$ by $\beta_i$, and projecting to $A$ yields
\[\beta_i\delta_i\alpha_i=\bigg(\sum_{\substack{a\in Q_1\\\head{a}=i}}a\beta_{\tail{a}}\delta_a\bigg)\alpha_i=\sum_{\substack{a\in Q_1\\\head{a}=i}}a\beta_{\tail{a}}\delta_a\alpha_{\head{a}}=0,\]
so that again $p\in\ker{\pi}$.

For the second statement, we have already shown that $\beta_{\tail{a}}\delta_a\alpha_{\head{a}}\in\ker{\pi}$ for any $a\in Q_1$ and $\beta_i\delta_i\alpha_i\in\ker{\pi}$ for any $i\in Q_0$. It then follows from the bipartite property of $\lift{F}$ that any path which is not in $\ker{\pi}$ and involves only arrows not in $Q_1$ must be a subpath of $\delta_x\alpha_i\beta_i\delta_y$ for some $x,y\in Q_0\cup Q_1$ and $i\in Q_0$ with $\head{x}=i=\tail{y}$. In particular, such a path has length at most $4$.
\end{proof}

\begin{thm}
\label{t:acyclic-fd}
Let $Q$ be an acyclic quiver. Then $A=\frjac{\lift{Q}}{\lift{F}}{\lift{0}}$ is finite-dimensional.
\end{thm}
\begin{proof}
We show that there are finitely many paths of $\lift{Q}$ determining non-zero elements of $A$. By Lemma~\ref{l:Q0-to-Q0}, any path $p$ of $\lift{Q}$ determining a non-zero element of $A$ may not have any subpath with endpoints in $Q_0$ and containing an arrow outside $Q_1$. Thus we must have $p=q_2p'q_1$, where $q_1$ and $q_2$ are (possibly empty) paths not involving arrows of $Q_1$, and $p'$ is a path in $Q$. Since $Q$ is acyclic, there are only finitely many possibilities for $p'$. By Lemma~\ref{l:Q0-to-Q0} again, $q_1$ and $q_2$ have length at most $4$ for $p$ to be non-zero in $A$, and so there are again only finitely many possibilities.
\end{proof}

We have now established everything we need in order to construct the Frobenius cluster category required by Theorem~\ref{t:main-cat-thm}, and to prove the first two parts of this theorem.

\begin{defn}
\label{d:pprin-cat}
Let $Q$ be an acyclic quiver, let $A_Q=\frjac{\lift{Q}}{\lift{F}}{\lift{W}}$ be the frozen Jacobian algebra defined from $Q$ in Definition~\ref{d:bigQP}, and let $B_Q=eA_Qe$ be its boundary algebra. We define $\frobcat_Q=\GP(B_Q)$ to be the category of Gorenstein projective $B_Q$-modules.
\end{defn}

\begin{proof}[Proof of Theorem~\ref{t:main-cat-thm}\ref{t:main-cat-thm-fcc}]
Since $Q$ is an acyclic quiver, $\jac{Q}{0}=\KK Q$ is finite-dimensional. By Theorem~\ref{t:acyclic-fd}, the algebra $A_Q=\frjac{\lift{Q}}{\lift{F}}{\lift{0}}$ is also finite-dimensional, and hence Noetherian.

Since the boundary algebra $B_Q=eA_Qe$ is finite-dimensional, Corollary~\ref{c:frobcat-constr} and Proposition~\ref{p:frob-clust-cat} combine to show that $\frobcat_Q=\GP(B_Q)$ is a Hom-finite Frobenius cluster category, as required.
\end{proof}

\begin{proof}[Proof of Theorem~\ref{t:main-cat-thm}\ref{t:main-cat-thm-cc}]
We use Keller--Reiten's recognition theorem \cite[Thm.~2.1]{kelleracyclic}; by Corollary~\ref{c:frobcat-constr}, $\stab{\frobcat}_Q$ is a $2$-Calabi--Yau triangulated category admitting a cluster-tilting object $T=eA_Q$ with endomorphism algebra $\KK Q$, so $\stab{\frobcat}_Q$ is triangle equivalent to the cluster category $\clustcat{Q}$ as required.
\end{proof}

We already observed after stating Corollary~\ref{c:frobcat-constr} that our construction applies to $(Q,W)$ given by the $3$-cycle with natural potential to produce a stably $2$-Calabi--Yau Frobenius cluster category $\frobcat_{Q,W}=\GP(B_{Q,W})$. We claim that the analogous statement to Theorem~\ref{t:main-cat-thm}\ref{t:main-cat-thm-cc} also holds here, namely that $\stab{\frobcat}_{Q,W}$ is triangle equivalent to Amiot's cluster category $\clustcat{Q,W}$. While Keller--Reiten's recognition theorem no longer applies to the cluster-tilting object $T$ provided by Corollary~\ref{c:frobcat-constr}, since its stable endomorphism algebra is not hereditary, one can in this case find another cluster-tilting object $T'\in\frobcat_{Q,W}$ with $\stabEndalg{\frobcat_{Q,W}}{T'}\iso\KK Q'$, for $Q'$ an orientation of the Dynkin diagram $\type{A}_{3}$. Thus the recognition theorem shows that $\stab{\frobcat}_{Q,W}\simeq\clustcat{Q'}$, and $\clustcat{Q'}\simeq\clustcat{Q,W}$ either by the recognition theorem again or by \cite[Cor.~3.11]{Amiot-ClustCat}.

We note that while Amiot--Reiten--Todorov have a recognition theorem \cite[Thm.~3.1]{amiotubiquity}, which identifies certain stable categories of Frobenius categories admitting a cluster-tilting object with frozen Jacobian endomorphism algebra as generalised cluster categories, this theorem does not apply to our object $T$ when $W\ne0$, because then $(\lift{Q},\lift{F},\lift{W})$ cannot satisfy their assumptions (H3) and (H4). Assumption (H3) requires a non-negative degree function on the arrows of $\lift{Q}$ giving $\lift{W}$ degree $1$, and (H4) requires that the arrows $\alpha_i$, for $i\in Q_0$, all have degree $1$. But if $W\ne0$, then (H3) forces some arrow $a\in Q_1$ to have degree $1$, and then (H4) implies that the potential term $a\beta_{\tail{a}}\delta_a\alpha_{\head{a}}$ has degree at least $2$, so (H3) does not hold.

\begin{rem}
\label{r:wu2}
We could also prove Theorem~\ref{t:main-cat-thm}\ref{t:main-cat-thm-cc} by using further results of Yilin Wu \cite[Thm.~6.2, Thm.~7.9]{Wu-Higgs}, exploiting in particular the fact that $A_{Q,W}$ is quasi-isomorphic to the relative Ginzburg dg-algebra $\Gamma_{\lift{Q},\lift{F},\lift{W}}$ (cf.~Remark~\ref{r:Ginzburg}). This does not increase the generality of the result, since this approach requires $A_{Q,W}$ to be finite-dimensional, a fact we only establish in the acyclic case (and which is false in general). However, it seems more realistic to generalise Wu's results to infinite-dimensional Jacobian algebras than to remove the acyclicity assumption from Keller--Reiten's theorem.
\end{rem}

\section{Mutation and the cluster character}
\label{s:mutation}

To complete the proof of Theorem~\ref{t:main-cat-thm}, we need to understand the relationship between mutation of cluster-tilting objects in $\GP(B_Q)$ and Fomin--Zelevinsky mutation of their quivers. These operations turn out to be compatible, in a much larger class of Frobenius cluster categories, in the sense of the following theorem.

\begin{thm}
\label{t:mutation}
Let $\frobcat$ be a Hom-finite Frobenius cluster category, and assume there is a cluster-tilting object $T\in\frobcat$ such that $\Endalg{\frobcat}{T}\cong\frjac{Q}{F}{W}$, where $Q$ the Gabriel quiver of this algebra. Then
\begin{enumerate}
\item\label{t:mutation-iy} $Q$ has no loops or $2$-cycles incident with any mutable vertex, so the Iyama--Yoshino mutation of $T$ at any indecomposable non-projective summand is well-defined,
\item\label{t:mutation-qpmut} if $T'$ is obtained from $T$ by mutation at such a summand, then $\End_\frobcat(T')^{\op}\cong\frjac{Q'}{F'}{W'}$, where $(Q',F',W')$ is the mutation of $(Q,F,W)$ at the vertex corresponding to this summand, and
\item\label{t:mutation-fzmut} the quiver $Q'$ is both the Gabriel quiver of $\End_\frobcat(T')^{\op}$ and, up to arrows between frozen vertices, the Fomin--Zelevinsky mutation of $Q$ at the appropriate vertex. In particular, it also has no loops or $2$-cycles incident with mutable vertices.
\end{enumerate}
These results then extend inductively to any cluster-tilting object in the mutation class of $T$.
\end{thm}
\begin{proof}
This combines \cite[Thm.~5.15]{Pressland-FJAs} and \cite[Prop.~3]{presslandcorrigendum} (which corrects \cite[Prop.~5.16]{Pressland-FJAs}).
\end{proof}

The proof of \cite[Thm.~5.15]{Pressland-FJAs} referred to in that of Theorem~\ref{t:mutation} is similar to an analogous result of Buan--Iyama--Reiten--Smith for triangulated categories \cite{BIRS2}. The definition of mutation for an ice quiver with potential is given in \cite[Def.~4.1]{Pressland-FJAs}, and is similar to that for ordinary quivers with potential \cite[\S1.2]{BIRS2}. We also explain in \cite{Pressland-FJAs} how to extend the definition of Fomin--Zelevinsky mutation to ice quivers which may have arrows between their frozen vertices.  Thus the same argument proves a stronger version of Theorem~\ref{t:mutation}\ref{t:mutation-fzmut} in which these arrows are still considered, but we will not use this here.

A consequence of Theorem~\ref{t:mutation} is that the Frobenius category $\frobcat_Q$ we attach to an acyclic quiver has a cluster structure in the sense of Buan--Iyama--Reiten--Scott \cite[\S II.1]{BIRS1}, which will allow us to use results of Fu and Keller \cite{FuKel} to complete the proof of Theorem~\ref{t:main-cat-thm}.

\begin{prop}
\label{p:clust-str}
The Frobenius category $\frobcat_Q$ from Definition~\ref{d:pprin-cat} has a cluster structure.
\end{prop}
\begin{proof}
By \cite[Thm.~II.1.6]{BIRS1}, it is enough to check that endomorphism algebras of cluster-tilting objects in $\frobcat_Q$ have no loops or $2$-cycles in their quivers. Because $\frobcat_Q$ is a Hom-finite Frobenius cluster category, this holds for those cluster-tilting objects in the mutation class of $T=eA_Q$ by Theorem~\ref{t:mutation}. But since $\stab{\frobcat}_Q\simeq\clustcat{Q}$ by Theorem~\ref{t:main-cat-thm}\ref{t:main-cat-thm-cc}, every cluster-tilting object of $\frobcat_Q$ is in this mutation class \cite[Thm.~A.1]{BMRT}.
\end{proof}

\begin{proof}[Proof of Theorem~\ref{t:main-cat-thm}\ref{t:main-cat-thm-clustvars}--\ref{t:main-cat-thm-mut}]
Since $\frobcat_Q$ is a Hom-finite stably $2$-Calabi--Yau Frobenius category, with cluster-tilting object $T=eA_Q$, it has a cluster character $\FKcc{T}\colon\Ob(\frobcat_Q)\to\pprinclust{Q}$ relative to $T$, as defined by Fu and Keller---see \cite[\S3]{FuKel}, or \eqref{eq:FK-cc} below, for an explicit formula for $\FKcc{T}$.

Moreover, in Fu--Keller's terminology \cite[Def.~5.1]{FuKel}, the pair $(\frobcat_Q,T)$ is a Frobenius $2$-Calabi--Yau realisation of the polarised principal coefficient cluster algebra $\pprinclust{Q}$; this follows from Proposition~\ref{p:clust-str} and by using Theorem~\ref{t:icytofrobcat}\ref{t:icytofrobcat-isos} to see that $\Endalg{\frobcat_Q}{T}\iso A_Q$ has ice quiver $\lift{Q}$.

Both Theorem~\ref{t:main-cat-thm}\ref{t:main-cat-thm-clustvars} and Theorem~\ref{t:main-cat-thm}\ref{t:main-cat-thm-mut} now follow from \cite[Thm.~5.4(a)]{FuKel}. Note that references to `reachable' objects in this theorem may be dropped since every indecomposable rigid object and every cluster-tilting object in $\frobcat_Q$ is reachable by \cite[Thm.~A.1]{BMRT}, as in the proof of Proposition~\ref{p:clust-str}. Note also that, unlike us, Fu--Keller do not consider the frozen variables of $\pprinclust{Q}$ to be cluster variables, which leads them to exclude the indecomposable projective objects from their statement. However, it is straightforward to check (see \cite[Thm.~3.3(a)]{FuKel}) that $\FKcc{T}$ also gives a bijection between the indecomposable projectives in $\frobcat_Q$, which are summands of $T$, and the frozen variables of $\pprinclust{Q}$.
\end{proof}

Just as for Theorem~\ref{t:main-cat-thm}\ref{t:main-cat-thm-cc}, the conclusions of Theorem~\ref{t:main-cat-thm}\ref{t:main-cat-thm-clustvars}--\ref{t:main-cat-thm-mut} still hold when $Q$ is a $3$-cycle, replacing $B_Q$ by $B_{Q,W}$ for $W$ the natural potential on $Q$. In this case $\pprinclust{Q}$ is the cluster algebra of the Grassmannian $G_2^6$, and $\GP(B_{Q,W})$ is the Grassmannian cluster category constructed by Jensen--King--Su \cite{JKS} (see Example~\ref{e:3-cycle}). The analogue of Theorem~\ref{t:mutation} holds for Grassmannian cluster categories, despite them being Hom-infinite---see \cite[Prop.~5.17]{Pressland-FJAs} and its preceding discussion.

For a general quiver with potential $(Q,W)$, even assuming that $A_{Q,W}$ is Noetherian so that the category $\frobcat_{Q,W}$ is well-defined, it may still be the case that the cluster-tilting objects in this category form several different mutation classes. However, this would not significantly affect our arguments here: while we can only rule out loops and $2$-cycles in the mutation class of $T=eA_{Q,W}$ via our methods, and so we cannot show that $\frobcat_{Q,W}$ has a cluster structure in the strict sense, Fu--Keller's bijections (with the word `reachable' reinstated) are still valid in this situation. A more detailed discussion of this issue can be found in \cite[\S6]{Pressland-Post}.

\section{An extriangulated categorification}
\label{s:extriangulated}

In this section we will prove Corollary~\ref{c:extriangulated}, but first we need to define the extriangulated category appearing in the statement. To keep our discussion brief, and since we only deal with very particular examples, we will not define extriangulated categories fully here, and instead refer the reader to Nakaoka and Palu \cite{nakaokaextriangulated}. The most important piece of data for us will be that of the `extension group' $\EE(X,Y)$ for any two objects $X$ and $Y$ in the extriangulated category; in our context, the objects of the category will be the same as those of the exact category $\frobcat_Q$, and we will have $\EE(X,Y)=\Ext^1_{\frobcat_Q}(X,Y)$.

\begin{defn}
\label{d:prin-cat}
Let $Q$ be an acyclic quiver, with associated Frobenius category $\frobcat_Q=\GP(B_Q)$ as in Definition~\ref{d:pprin-cat}. Write $e^-=\sum_{i\in Q_0}e_i^-$, and $P^-=B_{Q}e^-$. Define $\frobcat^+_{Q}=\frobcat_{Q}/[P^-]$ to be the quotient of $\frobcat_{Q}$ by the the ideal of morphisms factoring through an object of $\add{P^-}$. 
\end{defn}

\begin{prop}
\label{p:extriangulated}
The category $\frobcat^+_{Q}$ carries the structure of an extriangulated category. Moreover, it is a Frobenius extriangulated category \cite[Def.~7.1]{nakaokaextriangulated}, and its stable category $\stab{\frobcat}^+_{Q}$ coincides with $\stab{\frobcat}_{Q}$, meaning in particular that it is triangulated and $2$-Calabi--Yau.
\end{prop}
\begin{proof}
Since $P^-$ is a projective-injective object of $\GP(B_{Q})$, the extriangulated structure on $\frobcat_{Q}^+$ is induced from the exact structure on $\frobcat_{Q}$ as in \cite[Prop.~3.30]{nakaokaextriangulated}. Since the extension groups $\EE_{\frobcat^+_{Q}}(X,Y)$ in this extriangulated category coincide with the extension groups $\Ext^1_{\frobcat_{Q}}(X,Y)$ in the exact category $\frobcat_{Q}$, we see that an object is projective in $\frobcat^+_Q$ if and only if it is projective in $\frobcat_{Q}$, if and only if it is injective in $\frobcat_{Q}$, if and only if it is injective in $\frobcat^+_{Q}$. This agreement of extension groups also allows us to deduce the fact that $\frobcat^+_{Q}$ has enough projective and injective objects \cite[Def.~3.25]{nakaokaextriangulated} from this fact for $\GP(B_{Q})$, and thus conclude that $\frobcat^+_Q$ is Frobenius. It also follows from the agreement of projective-injective objects in $\frobcat^+_{Q}$ and $\frobcat_{Q}$ that $\stab{\frobcat}^+_{Q}=\stab{\frobcat}_{Q}$, the latter being a $2$-Calabi--Yau triangulated category by Corollary~\ref{c:frobcat-constr}.
\end{proof}

Note that if $(Q,W)$ is a Jacobi-finite quiver with potential such that $A_{Q,W}$ is Noetherian and internally Calabi--Yau with respect to $e$ (such as the $3$-cycle with its usual potential), so that we have a stably $2$-Calabi--Yau Frobenius category $\frobcat_{Q,W}$, we can define $\frobcat_{Q,W}^+=\frobcat_{Q,W}/[B_{Q,W}e^-]$. The statement of Proposition~\ref{p:extriangulated} also holds for this category, with the same proof.

%

\begin{proof}[Proof of Corollary~\ref{c:extriangulated}]
Statement \ref{c:extriangulated-cc} follows from Proposition~\ref{p:extriangulated}, using Theorem~\ref{t:main-cat-thm}\ref{t:main-cat-thm-cc} to see that $\stab{\frobcat}_Q^+=\stab{\frobcat}_Q$ is equivalent to the cluster category $\clustcat{Q}$.

To deduce the remaining statements from Theorem~\ref{t:main-cat-thm}, first note that
\[\Ext^1_{\frobcat_Q}(X,Y)=\EE_{\frobcat^+_Q}(X,Y)=\stabHom_{\frobcat_Q}(X,Y[1]).\]
for any $X$ and $Y$ in the common set of objects of the three categories $\frobcat_Q$, $\frobcat^+_Q$ and $\stab{\frobcat}_Q$. This means that a set of representatives of isomorphism classes of cluster-tilting objects in $\frobcat_Q$ is also such a set for $\frobcat_Q^+$ and for $\stab{\frobcat}_Q$. Moreover, since $\frobcat^+_Q$ and $\frobcat_Q$ have the same projective-injective objects, a direct summand of one of these cluster-tilting objects is indecomposable non-projective in one of these subcategories if and only if this is the case in all three. Thus two of these cluster-tilting objects are related by mutation at an indecomposable summand in one of the categories if and only if they are related by mutation at the same summand in all three. Consequently, the identity map is a bijection from the set of cluster-tilting objects of $\frobcat_Q$ to the set of cluster-tilting objects of $\frobcat^+_Q$, commuting with mutation. A more detailed discussion of mutation of cluster-tilting objects in extriangulated categories can be found in \cite{changcluster} (see also \cite{zhoutriangulated}).

Finally, the indecomposable rigid objects of $\frobcat_Q^+$ are, up to isomorphism, precisely the indecomposable rigid objects of $\frobcat_Q$ which are not summands of $P^-$.

It now follows from Theorem~\ref{t:main-cat-thm}\ref{t:main-cat-thm-clustvars}--\ref{t:main-cat-thm-mut} that composing Fu--Keller's cluster character $\Ob(\frobcat_Q)\to\pprinclust{Q}$ with the map $\pprinclust{Q}\to\princlust{Q}$ evaluating the frozen variables $x_i^-$ at $1$ provides the required bijections. To see Corollary~\ref{c:extriangulated}\ref{c:extriangulated-mut}, note additionally that each seed of $\princlust{Q}$ is obtained from one of $\pprinclust{Q}$ by removing the frozen variables $x_i^-$ and deleting the corresponding vertices (together with any incident arrows) from the quiver. Similarly, one obtains the quiver of $\Endalg{\frobcat_Q^+}{T}$ from that of $\Endalg{\frobcat_Q}{T}$ by removing the vertices corresponding to summands of $P^-$.
\end{proof}

\section{Boundary algebras}
\label{s:bdy-algs}

Since the objects of the Frobenius category $\frobcat_Q$ in Theorem~\ref{t:main-cat-thm} are modules for the idempotent subalgebra $B_Q=eA_Qe$ determined by the frozen vertices of $\lift{Q}$, we wish to describe this subalgebra more explicitly. In this section, we will give a presentation of $B_Q$ via a quiver with relations.

Recall that the double quiver $\double{Q}$ of a quiver $Q$ has vertex set $Q_0$ and arrows $Q_1\cup\dual{Q}_1$, where $\dual{Q}_1=\set{\dual{\alpha}:\alpha\in Q_1}$. The head and tail maps agree with those of $Q$ on $Q_1$, and are defined by $\head{\dual{\alpha}}=\tail{\alpha}$ and $\tail{\dual{\alpha}}=\head{\alpha}$ on $\dual{Q}_1$. The complete preprojective algebra of $Q$ is
\[\preproj_{Q}=\cpa{\KK}{\double{Q}}/\close{\langle\textstyle\sum_{\alpha\in Q_1}[\alpha,\dual{\alpha}]\rangle}\]
and, up to isomorphism, depends only on the underlying graph of $Q$. We begin with the following general statement for frozen Jacobian algebras, which reveals some of the relations of $B_{Q,W}$ for an arbitrary quiver with potential $(Q,W)$.

\begin{prop}
\label{p:universal-preproj}
Let $(Q,F,W)$ be an ice quiver with potential, let $A=\frjac{Q}{F}{W}$ and let $B=eAe$ be the boundary algebra of $A$. Then there is an algebra homomorphism $\pi\colon\preproj_{F}\to B$ given by $\pi(\idemp{i})=\idemp{i}$ for $i\in F_0$, and $\pi(\alpha)=\alpha$, $\pi(\dual{\alpha})=\der{\alpha}{W}$ for $\alpha\in F_1$.
\end{prop}
\begin{proof}
It suffices to check that $\pi(\sum_{\alpha\in F_1}[\alpha,\dual{\alpha}])=0$, i.e.\ that $\sum_{\alpha\in F_1}[\alpha,\der{\alpha}{W}]=0$.
By construction, for any $i\in Q_0$ we have
\[\sum_{\alpha\in\vin{i}}\alpha\der{\alpha}{W}-\sum_{\beta\in\vout{i}}\der{\beta}{W}\beta=0\]
in $\cpa{\KK}{Q}$. Projecting to $A$ and summing over vertices, we see that
\[0=\sum_{\alpha\in Q_1}[\alpha,\der{\alpha}{W}]=\sum_{\alpha\in Q_1^\mut}[\alpha,\der{\alpha}{W}]+\sum_{\alpha\in F_1}[\alpha,\der{\alpha}{W}]=\sum_{\alpha\in F_1}[\alpha,\der{\alpha}{W}],\]
where the final equality holds since $\der{\alpha}{W}=0$ in $A$ whenever $\alpha\in Q_1^\mut$.
\end{proof}

\begin{rem}
Familiarity with the constructions of \cites{GLS-PFVs,BIRS1,JKS} may make it tempting to conjecture that the map $\pi$ in Proposition~\ref{p:universal-preproj} is surjective, at least when $\frjac{Q}{F}{W}$ is bimodule internally $3$-Calabi--Yau, but this is in fact rarely the case. Indeed, we will see in this section that $\pi$ fails to be surjective for our frozen Jacobian algebras $\frjac{\lift{Q}}{\lift{F}}{\lift{W}}$ whenever $Q$ is acyclic with a path of length $2$; see Example~\ref{e:a3-bdy-eg} below.
\end{rem}

We now turn to our description of $B_Q$, beginning with what will turn out to be its Gabriel quiver.

\begin{defn}
Let $Q$ be an acyclic quiver, and consider the frozen subquiver $\lift{F}$ of $\lift{Q}$, which has vertex set $\lift{F}_0=\set{i^+,i^-:i\in Q_0}$ 
and arrows
\[
\delta_i\colon i^+\to i^-,\qquad
\delta_a\colon \head{a}^+\to\tail{a}^-
\]
for each $i\in Q_0$ and $a\in Q_1$. We define a quiver $\Lambda_Q$ by adjoining to $\lift{F}$ an arrow
\[\dual{\delta_p}\colon\tail{p}^-\to\head{p}^+\]
for each path $p$ of $Q$. Since $Q$ is acyclic, it has finitely many paths, and so $\Lambda_Q$ is a finite quiver.
\end{defn}

If $p=\idemp{i}$ is the trivial path at $i\in Q_0$, we write $\dual{\delta}_p=\dual{\delta}_i$ to avoid a double subscript. The double quiver of $\lift{F}$ appears as the subquiver of $\Lambda_Q$ obtained by excluding the arrows $\dual{\delta}_p$ for $p$ of length two or more; the notation for the arrows of $\Lambda_Q$ is chosen to be consistent with that used earlier for the arrows of this double quiver.

\begin{eg}
\label{e:a3-bdy-eg}
Let
\[Q=\mathord{\begin{tikzpicture}[baseline={([yshift=-1.2ex]current bounding box.center)}]
\node at (-1,0) (1) {$1$};
\node at (1,0) (2) {$2$};
\node at (3,0) (3) {$3$};
\path[-angle 90,font=\scriptsize]
	(1) edge node[above] {$a$} (2)
	(2) edge node[above] {$b$} (3);
\end{tikzpicture}}\]
be a linearly oriented quiver of type $\type{A}_3$. Then $\Lambda_Q$ is the following quiver.
\[\Lambda_Q=\mathord{\begin{tikzpicture}[baseline={([yshift=3.2ex]current bounding box.center)},xscale=1.5]
\node at (0,0) (1) {$1^+$};
\node at (1,0) (2) {$1^-$};
\node at (2,0) (3) {$2^+$};
\node at (3,0) (4) {$2^-$};
\node at (4,0) (5) {$3^+$};
\node at (5,0) (6) {$3^-$};
\path[-angle 90,font=\scriptsize]
	(1) edge[bend left] node[above] {$\delta_1$} (2)
	(2) edge[bend left] node[below] {$\dual{\delta}_1$} (1)
	(2) edge[bend right] node[below] {$\dual{\delta}_a$} (3)
	(3) edge[bend right] node[above] {$\delta_a$} (2)
	(3) edge[bend left] node[above] {$\delta_2$} (4)
	(4) edge[bend left] node[below] {$\dual{\delta}_2$} (3)
	(4) edge[bend right] node[below] {$\dual{\delta}_b$} (5)
	(5) edge[bend right] node[above] {$\delta_b$}  (4)
	(5) edge[bend left] node[above] {$\delta_3$} (6)
	(6) edge[bend left] node[below] {$\dual{\delta}_3$} (5)
	(2) edge[bend right=70] node[below] {$\dual{\delta}_{ba}$} (5);
\end{tikzpicture}}\]
\end{eg}

\begin{defn}
Let $Q$ be a quiver. A \emph{zig-zag} in $Q$ is a triple $(q,a,p)$, where $p$ and $q$ are paths in $Q$, and $a\in Q_1$ is an arrow such that $\head{p}=\head{a}$ and $\tail{q}=\tail{a}$.
Thus if $a\colon v\to w$ is an arrow of $Q$, a zig-zag involving $a$ is some configuration
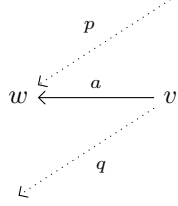
\[\begin{tikzpicture}[xscale=2,yscale=1.3] 
\node at (0,0) (w) {$w$};
\node at (1,0) (v) {$v$};
\path[-angle 90,font=\scriptsize]
	(v) edge node[above] {$a$} (w);
\path[-angle 90,font=\scriptsize,dotted]
	(1,1) edge node [above left] {$p$} (w)
	(v) edge node [below right] {$q$} (0,-1);
\end{tikzpicture}\]
where the dotted arrows denote paths. We call the zig-zag \emph{strict} if $p\ne ap'$ for any path $p'$ and $q\ne q'a$ for any path $q'$, but do not exclude these possibilities in general. If $z=(q,a,p)$ is a zig-zag, then we define $\head{z}=\head{q}$ and $\tail{z}=\tail{p}$.
\end{defn}

We now write down elements of $\cpa{\KK}{\Lambda_Q}$ that will turn out to be a set of generating relations for $B_Q$, having three flavours, as follows.
\begin{enumerate}[label=({R}\arabic*)]
\item\label{i:path-rel-t} For each path $p$ of $Q$, let
\[r_1(p)=\dual{\delta}_p\delta_{\tail{p}}-\sum_{\substack{a\in Q_1\\\head{a}=\tail{p}}}\dual{\delta}_{pa}\delta_a.\]
\item\label{i:path-rel-h} For each path $p$ of $Q$, let
\[r_2(p)=\delta_{\head{p}}\dual{\delta}_p-\sum_{\substack{a\in Q_1\\\tail{a}=\head{p}}}\delta_a\dual{\delta_{ap}}.\]
\item\label{i:zigzag-rel} For each zig-zag $(q,a,p)$ of $Q$, let
\[r_3(q,a,p)=\dual{\delta}_{q}\delta_a\dual{\delta}_{p}.\]
\end{enumerate}
We write $I$ for the the ideal of $\cpa{\KK}{\Lambda_Q}$ generated by the union of these three sets of relations. This generating set is usually not minimal, as for certain zig-zags $(q,a,p)$, the relation $r_3(q,a,p)$ may already lie in the ideal generated by relations of the form $r_1$ and $r_2$. For example, if $a\in Q_1$ is an arrow such that $i=\tail{a}$ is not incident with any other arrows of $Q$, then
\[
r_1(a)=\dual{\delta}_a\delta_i,\qquad
r_2(e_i)=\delta_i\dual{\delta_i}-\delta_a\dual{\delta}_a,
\]
so in $\cpa{\KK}{\Lambda_Q}/\close{\Span{r_1(a),r_2(e_i)}}$ we already have
\[r_3(a,a,a)=\dual{\delta_a}\delta_a\dual{\delta}_a=\dual{\delta}_a\delta_{i}\dual{\delta}_{i}=0.\]
One can check that if $(q,a,p)$ is a strict zig-zag, then $r_3(q,a,p)$ is not redundant, but this condition is not necessary; if $Q$ is a linearly oriented quiver of type $\type{A}_4$ and $a$ is the middle arrow, then the non-strict zig-zag $(a,a,a)$ yields an irredundant relation.

\begin{rem}
When $p=e_i$ is a vertex idempotent, the relations $r_1(e_i)$ and $r_2(e_i)$ reduce to the preprojective relations on the double quiver of $\lift{F}$, of the form predicted by Proposition~\ref{p:universal-preproj}. Each arrow $a\in Q_1$ is part of the trivial strict zig-zag $(e_{\tail{a}},a,e_{\head{a}})$, and so contributes an irredundant relation $r_3(e_{\tail{a}},a,e_{\head{a}})=\dual{\delta}_{\tail{a}}\delta_a\dual{\delta}_{\head{a}}$.
\end{rem}

Let $\Phi\colon\cpa{\KK}{\Lambda_Q}\to A_Q$ be the map given by the identity on the vertices of $\Lambda_Q$ and the arrows $\delta_i$ and $\delta_a$ for $i\in Q_0$ and $a\in Q_1$; this makes sense as these are subsets of the vertices and arrows of $\lift{Q}$. On the remaining arrows $\dual{\delta}_p$ of $\Lambda_Q$, we define $\Phi(\dual{\delta}_p)=\alpha_{\head{p}}p\beta_{\tail{p}}$.

\begin{prop}
\label{well-def}
The map $\Phi$ above induces a well-defined map $\Phi\colon\cpa{\KK}{\Lambda_Q}/\close{I}\to B_Q$.
\end{prop}
\begin{proof}
Since $\Phi$ sends every vertex or arrow of $\Lambda_Q$ to the image in $A_Q$ of a path in $\lift{Q}$ with frozen head and tail, it takes values in $B_Q$. It remains to check that it is zero on each of the generating relations of $I$, which we do by explicit calculation. Let $p$ be a path in $Q$. Then
\begin{align*}
\Phi(r_1(p))&=\alpha_{\head p}p\beta_{\tail{p}}\delta_{\tail{p}}-\sum_{\substack{a\in Q_1\\\head{a}=\tail{p}}}\alpha_{\head p}pa\beta_{\tail{a}}\delta_a=\alpha_{\head{p}}p(\der{\alpha_{\tail{p}}}{\lift{W}})=0,\\
\Phi(r_2(p))&=\delta_{\head{p}}\alpha_{\head p}p\beta_{\tail{p}}-\sum_{\substack{a\in Q_1\\\tail{a}=\head{p}}}\delta_a\alpha_{\head{a}}ap\beta_{\tail{p}}=(\der{\beta_{\head{p}}}{\lift{W}})p\beta_{\tail{p}}=0.
\end{align*}
If $(q,a,p)$ is a zig-zag, then
\[\Phi(r_3(q,a,p))=\alpha_{\head{r}}q\beta_{\tail{a}}\delta_a\alpha_{\head{a}}p\beta_{\tail{p}}=0,\]
since $0=\der{a}{\lift{W}}=\der{a}{W}-\beta_{\tail{a}}\delta_a\alpha_{\head{a}}=-\beta_{\tail{a}}\delta_a\alpha_{\head{a}}$ by acyclicity of $Q$.
\end{proof}

\begin{thm}
Let $Q$ be an acyclic quiver. Then the map $\Phi\colon\cpa{\KK}{\Lambda_Q}/\close{I}\to B_Q$, where $\Phi$, $\Lambda_Q$ and $I$ are all defined as above, is an isomorphism.
\end{thm}
\begin{proof}
We begin by showing surjectivity. As in the proof of Theorem~\ref{t:acyclic-fd}, we may use Lemma~\ref{l:Q0-to-Q0} to see that any path in $\lift{Q}$ determining a non-zero element of $A$ has the form $p=q_1p'q_2$ where $q_1$ and $q_2$ contain no arrows of $Q_1$, and $p'$ is a path of $Q$. If $p$ has frozen head and tail, then $q_1$ and $q_2$ must be non-zero, so we even have
\[p=q_1'\alpha_{\head{p'}}p'\beta_{\tail{p'}}q_2'=q_1'\Phi(\dual{\delta_{p'}})q_2'\]
Now $q_1'$ and $q_2'$ are, like $p$, paths of $\lift{Q}$ with frozen head and tail, but with the additional property that they include no arrows of $Q_1$. Let $q$ be such a path. If $q$ contains the arrow $\beta_i$ for some $i\in Q_0$, then this arrow must have a successor in $q$, since its head is unfrozen. Moreover, this successor must be $\alpha_i$, since this is the only arrow outside of $Q_1$ that may follow $\beta_i$. Similarly, any occurrence of $\alpha_i$ in $q$ must be preceded by $\beta_i$. It follows that $q$ is either a vertex idempotent $e_i^\pm=\Phi(e_i^\pm)$, or is formed by composing paths of the form $\delta_i=\Phi(\delta_i)$ for $i\in Q_0$, $\delta_a=\Phi(\delta_a)$ for $a\in Q_1$, or $\alpha_i\beta_i=\Phi(\dual{\delta}_i)$ for $i\in Q_0$, and so is in the image of $\Phi$. We conclude that the projection to $A_Q$ of any path in $\lift{Q}$ with frozen head and tail lies in the image of $\Phi$. Since these projections span $B_Q$, we see that $\Phi$ is surjective.

From now on, we abbreviate $B=B_Q$. To complete the proof, we will use \cite[Prop.~3.3]{BIRS2}, which in this context states that $\Phi$ is an isomorphism if and only if
\begin{equation}
\label{B-complex+}
\begin{tikzcd}[column sep=30pt]
\displaystyle\bigdsum_{\substack{\text{$p$ path in $Q$}\\\tail{p}=i}}Be_{\head{p}}^+\arrow{r}{f}
&Be_i^-\oplus\Big(\displaystyle\bigdsum_{\substack{a\in Q_1\\\head{a}=i}}Be_{\tail{a}}^-\Big)\arrow{r}{(\cdot\delta_i,
\cdot\delta_a)}&\rad{Be_i^+}\arrow{r}&0
\end{tikzcd}
\end{equation}
and
\begin{equation}
\label{B-complex-}
\begin{tikzcd}
\displaystyle\Big(\bigdsum_{\substack{\text{$p$ path in $Q$}\\\tail{p}=i}}Be_{\head{p}}^-\Big)\dsum\Big(\bigdsum_{\substack{\text{$z$ zig-zag}\\\tail{z}=i}}Be_{\head{z}}^+\Big)\arrow{r}{g}
&\displaystyle\bigdsum_{\substack{\text{$p$ path in $Q$}\\\tail{p}=i}}Be_{\head{p}}^+\arrow{r}{\cdot\Phi(\dual{\delta}_p)}&\rad{Be_i^-}\arrow{r}&0
\end{tikzcd}
\end{equation}
are exact sequences for all $i\in Q_0$. Here the left-most maps in each sequence are obtained from our generators for the ideal $I$ by right-differentiation \eqref{eq:rightder} and the application of $\Phi$, as prescribed in \cite{BIRS2}, so they act on components by
\begin{equation}
\label{rel-maps}
\begin{aligned}
f(ye_{\head{p}}^+)&=y\Phi(\dual{\delta}_p)\tensor[i]-\sum_{\substack{a\in Q_1\\\head{a}=i}}y\Phi(\dual{\delta}_{pa})\tensor[a],\\
g(ye_{\head{p}}^-)&=y\delta_{\head{p}}\tensor[p]-\sum_{\substack{a\in Q_1\\\tail{a}=\head{p}}}y\delta_a\tensor[ap]\\
g(ye_{\head{z}}^+)&=y\Phi(\dual{\delta}_{q})\delta_a\otimes[p],\ \text{where $z=(q,a,p)$.}
\end{aligned}
\end{equation}
Here we have dealt with the ambiguity about which summand contains each term as in Section~\ref{s:frjacCY}, denoting elements of $\bigdsum_{a\in X}Be_{v(a)}$ by $\sum_{a\in X}x_a\otimes[a]$, with $x_a\in Be_{v(x)}$ and $x_a\otimes[a]$ denoting its inclusion into the summand of $\bigdsum_{a\in X}Be_{v(a)}$ indexed by $a$. We adopt the convention here that the summand $B\idemp{i}^-$ appearing in the middle term of \eqref{B-complex+} is indexed by the vertex $i$.

Since $\Phi$ is well-defined and surjective, sequences \eqref{B-complex+} and \eqref{B-complex-} are complexes and exact at $\rad{Be_i^+}$ and $\rad{Be_i^-}$ respectively, so we need only check exactness at the middle term in each case. We proceed as in Section~\ref{s:frjacCY}, using the explicit set $R$ of relations for $A_Q$ from \eqref{eq:relations}, and begin with \eqref{B-complex+}. Let $x_i\in e\cpa{\KK}{\lift{Q}}e_i^-$ and $x_a\in e\cpa{\KK}{\lift{Q}}e_{\tail{a}}^-$ for each $a\in Q_1$ with $\head{a}=i$, determining the element
\[x=x_i\otimes[i]+\sum_{\substack{a\in Q_1\\\head{a}=i}}x_a\otimes[a]\]
of the middle term of \eqref{B-complex+}. Assume that $x$ is a cycle for this complex, i.e.\ that
\[x_i\delta_i+\sum_{\substack{a\in Q_1\\\head{a}=i}}x_a\delta_a\in\close{\Span{R}}.\]
Since the only generating relation for $A_Q$ with terms ending in $\delta_i$ or $\delta_a$ for $a\in Q_1$ with $\head{a}=i$ is $\der{\alpha_i}{W}$, it follows that in $\cpa{\KK}{\lift{Q}}$ we have
\begin{align*}
x_i\delta_i+\sum_{\substack{a\in Q_1\\\head{a}=i}}x_a\delta_a&=z_i\delta_i+\sum_{\substack{a\in Q_1\\\head{a}=i}}z_a\delta_a+y\der{\alpha_i}{W}\\
&=z_i\delta_i+\sum_{\substack{a\in Q_1\\\head{a}=i}}z_a\delta_a+y\bigg(\beta_i\delta_i-\sum_{\substack{a\in Q_1\\\head{a}=i}}a\beta_{\tail{a}}\delta_a\bigg)\end{align*}
for $z_i,z_a\in\close{\Span{R}}$ and $y\in\cpa{\KK}{\lift{Q}}\idemp{i}$. Comparing terms, we see that $x_i=z_i+y\beta_i$ and $x_a=z_a-ya\beta_{\tail{a}}$.
Since $\head{x_i}$ and $\head{x_a}$ are frozen, but $\head{\beta_j}$ is unfrozen for all $j\in Q_0$, we must have
\[y=\sum_{\substack{\text{$p$ path in $Q$}\\\tail{p}=i}}y_p\alpha_{\head p}p.\]
for some $y_p\in Be_{\head{p}}^+$. Projecting to $B$, we have
\[
x_i=\sum_{\substack{\text{$p$ path in $Q$}\\\tail{p}=i}}y_p\Phi(\dual{\delta}_p),\qquad
x_a=\sum_{\substack{\text{$p$ path in $Q$}\\\tail{p}=i}}-y_p\Phi(\dual{\delta}_{pa}),
\]
Thus 
\[y=\sum_{\substack{\text{$p$ path in $Q$}\\\tail{p}=i}}y_p\otimes p\]
satisfies $f(y)=x$, and so sequence \eqref{B-complex+} is exact.

Now we turn to \eqref{B-complex-}. For each path $p$ in $Q$ with $\tail{p}=i$, pick $x_p\in e\cpa{\KK}{\lift{Q}}e_{\head{p}}^+$, and assume that
\[\sum_{\substack{\text{$p$ path in $Q$}\\\tail{q}=i}}x_p\alpha_{\head{p}}p\beta_i\in\close{\Span{R}},\]
so that
\[x=\sum_{\substack{\text{$p$ path in $Q$}\\\tail{p}=i}}x_p\otimes[p]\]
is a cycle in the middle term of \eqref{B-complex-}.
By comparison with the generating relations, we see that we may write
\[\sum_{\substack{\text{$p$ path in $Q$}\\\tail{p}=i}}x_p\alpha_{\head{p}}p\beta_i=\sum_{\substack{\text{$p$ path in $Q$}\\\tail{p}=i}}\Bigg(z_p\alpha_{\head{p}}p\beta_i+y_p\bigg(\delta_{\head{p}}\alpha_{\head{p}}-\sum_{\substack{b\in Q_1\\\tail{b}=\head{p}}}\delta_b\alpha_{\head{b}}b\bigg)p\beta_i-\sum_{\substack{a\in Q_1\\\head{a}=\head{p}}}y_{a,p}(\beta_{\tail{a}}\delta_a\alpha_{\head{a}})p\beta_i\Bigg)\]
for some $z_p\in\close{\Span{R}}$, $y_p\in\cpa{\KK}{\lift{Q}}\idemp{\head{p}}^-$ and $y_{a,p}\in e\cpa{\KK}{\lift{Q}}\idemp{\tail{a}}$. Each path $p$ in the sum is either the trivial path $e_i$ or of the form $p=br$ for some arrow $b$ and path $r$. By comparing terms, we deduce that after projection to $B$ we have
\[
x_i=y_i\delta_i-\sum_{\substack{a\in Q_1\\\head{a}=i}}y_{a,e_i}\beta_{\tail{a}}\delta_a,\qquad
x_{br}=y_{br}\delta_{\head{b}}-y_{r}\delta_b-\sum_{\substack{a\in Q_1\\\head{a}=\head{b}}}y_{a,br}\beta_{\tail{a}}\delta_a,
\]
Since $y_{a,p}\in e\cpa{\KK}{\lift{Q}}\idemp{\tail{a}}$, we must have
\[y_{a,p}=\sum_{\substack{\text{$q$ path in $Q$}\\\tail{q}=\tail{a}}}y_{q,a,p}\alpha_{\head{q}}q\]
for some $y_{q,a,p}\in e\KK \lift{Q}e_{\head{q}}^+$. The triple $z=(q,a,p)$ occurring in a subscript here satisfies $\head{a}=\head{q}$ and $\tail{a}=\tail{p}$, so it is a zig-zag. One may then calculate explicitly using \eqref{rel-maps} that the $y_p$ and $y_z=y_{q,a,p}$ give a preimage
\[y=\sum_{\substack{\text{$p$ path in $Q$}\\\tail{p}=i}}y_p\otimes[p]+\sum_{\substack{\text{$z$ zig-zag}\\\tail{z}=i}}y_z\otimes[z]\]
of $x$ under $g$.
\end{proof}

We close this section with the following curious property of the category $\frobcat_Q=\GP(B_Q)$. Recall that $\Sub(X)$, for $X\in\fgmod{B_Q}$, denotes the full subcategory of $\fgmod{B_Q}$ on objects $M$ admitting a monomorphism $M\to X^n$ for some $n$.

\begin{prop}
Write $P^+=\bigdsum_{i\in Q_0}B_Qe_k^+$. If $Q$ is acyclic and has no isolated vertices, then $\GP(B_Q)\subseteq\Sub(P^+)$.
\end{prop}
\begin{proof}
Write $B=B_Q$. Since $\GP(B)$ is a Frobenius category with injective objects those in $\add{B}$, it suffices to show that $B\in\Sub(P^+)$, or that $Be_i^\pm\in\Sub(P^+)$ for each $i$. Since this is true of $Be_i^+$ by definition of $P^+$, it remains to show $Be_i^-\in\Sub(P^+)$ for each $i\in Q_0$. Since $Q$ has no isolated vertices, $i$ cannot be both a source and a sink.

First assume $i$ is not a source in $Q$. Then the map $Be_i^-\to Be_i^+$ given by right multiplication by $\delta_i$ is injective, as follows. If $x\in Be_i^-$ satisfies $x\delta_i=0$ in $Be_i^+$, then lifting to $\cpa{\KK}{\Lambda_Q}$ and using the explicit generating set of $I$, we have
\[x\delta_i=z\delta_i+\sum_{\substack{\text{$p$ path in $Q$}\\\tail{p}=i}}y_p\bigg(\dual{\delta}_p\delta_i-\sum_{\substack{a\in Q_1\\\head{a}=i}}\dual{\delta_{pa}}\delta_a\bigg)\]
for some $z\in I$ and $y_p\in e\cpa{\KK}{\Lambda_Q}\idemp{\head{p}}^+$. Since $i$ is not a source, the sum over arrows on the right-hand side is non-empty. By comparing coefficients we see that $y_p=0$ for all $p$, and hence $x=z=0$ in $A$.

Now assume $i$ is not a sink. Pick $a\in Q_1$ with $\tail{a}=i$. Then the map $Be_i^-\to Be_{\head{a}}^+$ given by right multiplication by $\delta_a$ is injective as follows. If $x\in Be_i^-$ satisfies $x\delta_a=0$ in $Be_{\head{a}}^+$, then lifting to $\cpa{\KK}{\Lambda_Q}$ and using our explicit relations, we have
\[x\delta_a=z+\sum_{\substack{\text{$p$ path in $Q$}\\\tail{p}=\head{a}}}y_p\bigg(\dual{\delta_p}\delta_{\head{a}}-\sum_{\substack{b\in Q_1\\\head{b}=\head{a}}}\dual{\delta}_{pa}\delta_a\bigg),\]
for $z\in I$, and it follows by comparing coefficients that $y_p=0$ for all $p$, so $x=0$.
\end{proof}
Example~\ref{e:a3-eg} below shows that we may have $\frobcat_Q\subsetneq\Sub(P^+)$. The assumption on isolated vertices is necessary, since such a vertex $i$ of $Q$ results in a direct summand of $\frobcat_Q$ equivalent to $\fgmod{\Pi}$ for $\Pi$ the preprojective algebra of type $\type{A}_2$, with indecomposable objects $S_i^\pm$ and $P_i^\pm$, and $S_i^-,P_i^-\notin\Sub(P^+)=\Sub(P_i^+)$.

\section{g-vectors, indices and dimension vectors}
\label{s:index}

In this section we show how representation theoretic information about the category $\frobcat_Q$ from Theorem~\ref{t:main-cat-thm} can be used to recover information about the principal coefficient cluster algebra $\princlust{Q}$. These are not new results (unsurprisingly, since this cluster algebra is very well-studied), but the style of proofs is different and demonstrates how the categorification given here can be used. Our results are also not restricted to the acyclic case, but hold whenever an appropriate Frobenius categorification can be constructed, for example in the context of Corollary~\ref{c:frobcat-constr}.

Let $(Q,W)$ be a Jacobi-finite quiver with potential such that the frozen Jacobian algebra $A=A_{Q,W}$ is Noetherian; for example, one can take $Q$ to be an acyclic quiver. Then by Corollary~\ref{c:frobcat-constr}, the category $\frobcat=\GP(eAe)$ is a stably $2$-Calabi--Yau category with a cluster-tilting object $T=eA$, and there are compatible isomorphisms $\Endalg{\frobcat}{T}\iso A$ and $\stabEndalg{\frobcat}{T}\iso\jac{Q}{W}$. We keep the abbreviated notation $A$, $T$ and $\frobcat$ for the rest of the section.

The isomorphism $\Endalg{\frobcat}{T}\iso A$ allows us to write the indecomposable summands of $T$ as $T_i$, for $i\in Q_0$, where $T_i$ corresponds to $Ae_i$ under the isomorphism. We write $P_i^\pm$ instead of $T_i^\pm$, noting that these are the indecomposable projective-injective objects of $\frobcat$, and write $P^\pm=\bigdsum_{i\in Q_0}P_i^\pm$, which is compatible with the notation in Section~\ref{s:bdy-algs}. We abbreviate $F=\Hom_{\frobcat}(T,\blank)$ and $G=\Ext^1_{\frobcat}(T,\blank)$, these being functors $\frobcat\to\fgmod{A}$ via the isomorphism $\Endalg{\frobcat}{T}\iso A$. Finally, we write $[M:N]$ for the multiplicity of $N$ as an indecomposable summand of $M$ (for $M$ and $N$ objects in any category in which this makes sense).

\begin{lem}
\label{l:GX-proj-res}
Let $X\in\frobcat$. Then the minimal projective resolution
\begin{equation}
\label{eq:G-proj-res}
\begin{tikzcd}
0\arrow{r}&P^3\arrow{r}&P^2\arrow{r}&P^1\arrow{r}&P^0\arrow{r}&GX\arrow{r}&0
\end{tikzcd}
\end{equation}
of $GX$ has the property that, for each $j\in Q_0$, 
\begin{align*}
[P^k:Ae_j^+]&={\begin{cases}\dim_\KK(e_jGX),&k=1,\\0,&\text{otherwise},\end{cases}}\\
[P^k:Ae_j^-]&={\begin{cases}\dim_\KK(e_jGX)&k=2,\\0,&\text{otherwise}.\end{cases}}
\end{align*}
\end{lem}
\begin{proof}
Note that $eGX=\Ext^1_\frobcat(P^+\oplus P^-,X)=0$, so the composition series of $GX$ includes only the simple modules $S_i$ for $i\in Q_0$ a mutable vertex. In this case, the sequence \eqref{eq:proj-res-simp-i}, which is exact by Propositions~\ref{p:mut-exact-3} and \ref{p:mut-exact-2}, is a projective resolution of the simple $A$-module $\simp{i}$. We write this resolution as
\begin{equation}
\label{eq:simp-proj-res}
\begin{tikzcd}
0\arrow{r}&P_i^3\arrow{r}&P_i^2\arrow{r}&P_i^1\arrow{r}&P_i^0\arrow{r}&S_i\arrow{r}&0
\end{tikzcd}
\end{equation}
for projective $A$-modules $P_i^k$ which can be described explicitly using \eqref{eq:resA-isos}. In particular, one sees using this description that the only appearances of the projective $A$-modules $Ae_j^\pm$ in \eqref{eq:simp-proj-res} are one copy of $Ae_i^+$ in $P_i^1$, and one copy of $Ae_i^-$ in $P_i^2$.

By the horseshoe lemma, we can thus construct a (possibly non-minimal) projective resolution of $GX$ in which the degree $k$ term $\hat{P}^k$ is a direct sum of the $P_i^k$ over composition factors $S_i$ of $GX$. As a consequence, this resolution includes $\dim_\KK(e_jGX)$ copies of $Ae_j^+$, all summands of $\hat{P}^1$, and $\dim_\KK(e_jGX)$ copies of $Ae_j^-$, all summands of $\hat{P}^2$. Since each $Ae_j^\pm$ appears in only one homological degree in this resolution, it cannot be part of an acyclic summand, and so it appears with the same multiplicity and degree in the minimal projective resolution \eqref{eq:G-proj-res} of $GX$, as required.
\end{proof}

Recall that any $X\in\frobcat$ fits into an exact sequence
\begin{equation}
\label{eq:index}
\begin{tikzcd}
0\arrow{r}&K_X\arrow{r}&R_X\arrow{r}&X\arrow{r}&0
\end{tikzcd}
\end{equation}
with $K_X,R_X\in\add{T}$. Indeed, this holds for any Frobenius cluster category and cluster-tilting object, and the choice of sequence \eqref{eq:index} is equivalent to the choice of the map $R_X\to X$, subject to the condition that this map is a right $\add{T}$-approximation. Dually, $X$ fits into an exact sequence
\begin{equation}
\label{eq:coindex}
\begin{tikzcd}
0\arrow{r}&X\arrow{r}&L_X\arrow{r}&C_X\arrow{r}&0
\end{tikzcd}
\end{equation}
with $L_X,C_X\in\add{T}$, with the choice of sequence being equivalent to the choice of left $\add{T}$-approximation $X\to L_X$. We say the sequence \eqref{eq:index} is \emph{minimal} if the map $R_X\to X$ is a minimal right $\add{T}$-approximation of $X$, or equivalently if no non-zero component of the map $K_X\to R_X$ is an isomorphism, and define minimality for \eqref{eq:coindex} dually.

\begin{lem}
\label{l:index-sequence}
For any $X\in\frobcat$, choose minimal sequences of the form \eqref{eq:index} and \eqref{eq:coindex}, and let $i\in Q_0$. Then
\begin{align*}
[K_X:P_i^\pm]=[C_X:P_i^\pm]&=0,\\
[R_X:P_i^+]=[L_X:P_i^-]&=0,\\
[R_X:P_i^-]=[L_X:P_i^+]&=\dim_\KK(e_iGX).
\end{align*}
\end{lem}

\begin{proof}
The first line of equalities follows immediately from minimality, using that $P_i^\pm$ is projective-injective. For the remaining equalities, recall that $F$ restricts to an equivalence $\add{T}\isoto\projcat{A}$ by the Yoneda lemma. Applying $F$ to \eqref{eq:index} yields the exact sequence
\begin{equation}
\label{eq:F-proj-res}
\begin{tikzcd}
0\arrow{r}&FK_X\arrow{r}&FR_X\arrow{r}&FX\arrow{r}&0,
\end{tikzcd}
\end{equation}
which is a projective resolution of $FX$. Applying $F$ to \eqref{eq:index} yields
\[\begin{tikzcd}
0\arrow{r}&FX\arrow{r}&FL_X\arrow{r}&FC_X\arrow{r}&GX\arrow{r}&0,
\end{tikzcd}\]
and taking the cup product with \eqref{eq:F-proj-res} yields
\begin{equation}
\label{eq:G-proj-res2}
\begin{tikzcd}
0\arrow{r}&FK_X\arrow{r}&FR_X\arrow{r}&FL_X\arrow{r}&FC_X\arrow{r}&GX\arrow{r}&0,
\end{tikzcd}
\end{equation}
which is a projective resolution of $GX$. Furthermore, since the $\add{T}$-approximations $R_X\to X$ and $X\to L_X$ were chosen to be minimal, the projective resolutions \eqref{eq:F-proj-res} and \eqref{eq:G-proj-res2} are also minimal.  Noting that $[R_X:P_i^\pm]=[FR_X:FP_i^\pm]$ and $[L_X:P_i^\pm]=[FL_X:FP_i^\pm]$, and that $FP_i^\pm=Ae_i^\pm$ by Corollary~\ref{c:frobcat-constr}, the result follows from Lemma~\ref{l:GX-proj-res}.
\end{proof}


These observations will allow us to compute the g-vectors of any cluster algebra arising from a seed with principal part $Q$ in terms of in terms of the category $\frobcat=\frobcat_Q$. We define g-vectors via a grading on the principal coefficient cluster algebra $\princlust{Q}$, as in \cite[\S6]{FZ-CA4}. Write the initial cluster variables of this cluster algebra as $x_i$ and $x_i^+$ for $i\in Q_0$, where the $x_i^+$ are frozen. We then give these variables $\ZZ^{Q_0}$-degrees
\[\deg{x_i}=\varepsilon_i,\qquad \deg{x_i^+}=b_i,\]
where $\varepsilon_i$ is the $i$-th standard basis vector (named to avoid confusion with idempotents) and $b_i$ is the $i$-th row of the exchange matrix $b$ of $Q$. We extend this to the initial cluster variables of $\pprinclust{Q}$ by setting $\deg{x_i^-}=0$. This assignment puts a grading on the Laurent polynomial ring in variables $x_i$, $x_i^+$ and $x_i^-$, and thus on its subalgebra $\pprinclust{Q}$; note that the specialisation $x_i^-=1$, which restricts to a map $\pprinclust{Q}\to\princlust{Q}$, preserves degrees. We refer to the degree of any homogeneous element of the Laurent polynomial ring in $x_i$ and $x_i^+$ as the \emph{g-vector} of the element.

All cluster variables of $\pprinclust{Q}$ are homogeneous in the above grading; by work of Grabowski \cite[Prop.~3.2]{grabowskigraded}, this follows from the matrix identity
\begin{equation}
\label{eq:prin-grad}
\begin{pmatrix}-b&1_n&-1_n\end{pmatrix}\begin{pmatrix}1_n\\b\\0_n\end{pmatrix},
\end{equation}
where $1_n$ is the $n\times n$ identity matrix and $0_n$ is the $n\times n$ zero matrix, by observing that the first matrix in the product is the transposed exchange matrix of the initial seed of $\pprinclust{Q}$, and the second has rows given by the degrees of the cluster variables in this seed. Thus the cluster variables of $\princlust{Q}$ are also homogeneous and hence have well-defined g-vectors.

Given any seed $s$ with principal part $Q$, the mutable cluster variables of $\princlust{Q}$ are in natural bijection with those of the cluster algebra $\clustalg{s}$ with initial seed $s$, and so we define the g-vector of a cluster variable in $\clustalg{s}$ to be the g-vector of the corresponding cluster variable in $\princlust{Q}$. This applies in particular to the coefficient-free cluster algebra $\clustalg{Q}$.

The main results of this section involve the Fu--Keller cluster character $\FKcc{T}$ as in \cite{FuKel} and Section~\ref{s:mutation}. While we are only assuming that $A=\Endalg{\frobcat}{T}$ is Noetherian, and not necessarily finite-dimensional, this is sufficient for the cluster character to be well-defined (cf.~\cite[\S3]{grabowskigradedfrobenius} and \cite[\S6]{Pressland-Post}).

\begin{thm}
\label{t:g-vector-index}
Assume that $A=A_{Q,W}$ is Noetherian and consider the corresponding Frobenius category $\frobcat=\frobcat_{Q,W}$ and cluster-tilting object $T=eA$. Let $X\in\frobcat$, and choose a minimal short exact sequence of the form \eqref{eq:index}. Write
\[
R_X=P\dsum\bigdsum_{i\in Q_0}T_i^{r_i},\qquad
K_X=\bigdsum_{i\in Q_0}T_i^{k_i}.
\]
where $P\in\frobcat$ is projective and $T_i=eAe_i$. Then $\FKcc{T}(X)$ is homogeneous, and its g-vector is $\sum_{i\in Q_0}(r_i-k_i)\varepsilon_i$.
\end{thm}
\begin{proof}
We first observe that $K_X$ has no projective summands, since this would violate minimality, and so $R_X$ and $K_X$ have expressions of the required form. The element
\[G=\sum_{i\in Q_0}[S_i]\otimes \varepsilon_i+\sum_{i\in Q_0}[S_i^+]\otimes b_i\in\Kgp_0(\fd{A})\otimes\ZZ^{Q_0}\]
where $S_j$ denotes the simple $A$-module at vertex $j\in\lift{Q}_0$, is a grading of $\frobcat$ in the sense of \cite[Def.~3.7]{grabowskigradedfrobenius}; this follows again from the matrix identity \eqref{eq:prin-grad}, as in \cite[Rem.~3.9(ii)]{grabowskigradedfrobenius}. We write this grading of $\frobcat$ as $\deg_G$. It then follows from \cite[Prop.~3.11(i)]{grabowskigradedfrobenius} that $\FKcc{T}(X)$ is homogeneous and its g-vector is equal to $\deg_G(X)$.

Now by \cite[Prop.~3.11(ii)]{grabowskigradedfrobenius}, we have
\[\deg_G(X)=\deg_G(R_X)-\deg_G(K_X).\]
By construction, we have
\begin{align*}
\deg_G(R_X)&=\deg_G(P)+\sum_{i\in Q_0}r_i\deg_G(T_i)=\deg_G(P)+\sum_{i\in Q_0}r_i\varepsilon_i,\\
\deg_G(K_X)&=\sum_{i\in Q_0}k_i\deg_G(T_i)=\sum_{i\in Q_0}k_i\varepsilon_i,
\end{align*}
so that
\[\deg_G(X)=\deg_G(P)+\sum_{i\in Q_0}(r_i-k_i)\varepsilon_i.\]
But by Lemma~\ref{l:index-sequence}, we have $P\in\add{P^-}$, so $\deg_G(P)=0$ and the result follows.
\end{proof}
\begin{rem}
\label{r:index-gvec}
The quantity $\sum_{i\in Q_0}(r_i-k_i)\varepsilon_i$ is by definition (cf.~\cite[\S2.1]{palucluster}) the index of $X$ with respect to $T$. Thus Theorem~\ref{t:g-vector-index} recovers the well-known result \cite[Prop.~4.3]{FuKel} that the g-vector is categorified by the index, at least for those cluster algebras admitting appropriate additive categorifications (e.g.\ those defined by acyclic quivers).
\end{rem}

We may use similar methods to show how the cluster character $\FKcc{T}$ on $\frobcat$ may be computed solely in terms of the representation theory of the stable category $\stab{\frobcat}$, and of the finite-dimensional algebra $\jac{Q}{W}$. This is the first point at which we use an explicit formula for $\FKcc{T}$, and so we recall this from \cite{FuKel}.

Recall that we are assuming $\Endalg{\frobcat}{T}\iso A=A_{Q,W}$ is Noetherian. A consequence of Theorem~\ref{t:bi3cy} is that $\gldim{A}\leq 3$, and so the Euler form
\[\langle M,N\rangle=\sum_{i=0}^3(-1)^i\dim\Ext^i_A(M,N)\]
is well-defined when $M\in\fgmod{A}$ is finitely generated and $N\in\fgmod{A}$ is finite-dimensional. Since $A$ has finite global dimension, the value of $\langle M,N\rangle$ depends only on the dimension vector of the finite-dimensional module $N$ by the horseshoe lemma, and so we may write $\langle M,d\rangle:=\langle M,N\rangle$ for an arbitrary $A$-module $N$ with dimension vector $d\in\ZZ^{\lift{Q}_0}$. Theorem~\ref{t:bi3cy} also implies that $\langle M,N\rangle=-\langle N,M\rangle$ if both modules are finite-dimensional and $eN=0$, so we may analogously define $\langle d,M\rangle$ for $d\in\ZZ^{\lift{Q}_0}$ supported only on the set $Q_0$ of mutable vertices and $M$ a finite-dimensional $A$-module.

Recall further that we write $F=\Hom(T,\blank)$ and $G=\Ext^1(T,\blank)$. Now Fu--Keller's cluster character for $X\in\frobcat$ can be written
\begin{equation}
\label{eq:FK-cc}
\FKcc{T}(X)=x^{\hat{g}_X}\sum_{d\in\ZZ^{Q_0}}\lambda_dx^{\hat{v}_d},
\end{equation}
with notation as follows. The vector $\hat{g}_X\in\ZZ^{\lift{Q}_0}$ has $i$-th coordinate $\langle FX,S_i\rangle$, where $FX=\Hom_{\frobcat}(T,X)$ as above and $S_i$ denotes the simple $A$-module at vertex $i\in\lift{Q}_0$. Similarly, the vector $\hat{v}_d\in\ZZ^{\lift{Q}_0}$ has $i$-th coordinate $\langle d,S_i\rangle=-\langle S_i,d\rangle$. We use the standard monomial notation $x^u:=\prod_{i\in\lift{Q}_0}x_i^{u_i}$ for $u\in\ZZ^{\lift{Q}_0}$. Finally, $\lambda_d\in\ZZ$ is the Euler characteristic of the quiver Grassmannian of $d$-dimensional submodules of the $A$-module $GX=\Ext^1_\frobcat(T,X)$. Note that $GX$ is a finite-dimensional $\stabEndalg{\frobcat}{T}\iso\jac{Q}{W}$-module, hence the restriction to $d\in\ZZ^{Q_0}$ in the summation. Thus $\lambda_d$ can be non-zero only for the finite set of $d\in\ZZ^{Q_0}$ satisfying $0\leq d\leq\dimvec{GX}$ coordinate-wise.

For each $X\in\frobcat$ and $d\in\ZZ^{\lift{Q}_0}$, let $g_X,v_d\in\ZZ^{Q_0}$ be the restrictions of $\hat{g}_X$ and $\hat{v}_d$ respectively to the mutable vertices $Q_0$. Further, write
\begin{equation}
\label{eq:stab-cc}
\Palcc{T}(X)=\FKcc{T}(X)|_{x_i^\pm=1}.
\end{equation}
By \cite[Thm.~3.3(b)]{FuKel}, this function coincides with Palu's cluster character \cite{palucluster} for the stable category $\stab{\frobcat}$, evaluated on the shifted object $\Sigma X$. Consequently, the data of $g_X$, $v_d$ and $\lambda_d$ may be computed entirely from $\stab{\frobcat}$ and the algebra $\jac{Q}{W}$.

The next result shows that this is in fact also true of the extended vectors $\hat{g}_X$ and $\hat{v}_d$, and hence of the cluster character $\FKcc{T}(X)$. To state it, we extend our monomial notation to write $(x^+)^u=\prod_{i\in Q_0}(x_i^+)^{u_i}$ when $u\in\ZZ^{Q_0}$, and define $(x^-)^u$ similarly. We treat $\ZZ^{Q_0}$ as a subset of $\ZZ^{\lift{Q}_0}$ in the natural way, so that a monomial $x^u$ for $u\in\ZZ^{Q_0}$ lies in the same ring as a monomial $x^v$ for $v\in\ZZ^{\lift{Q}_0}$ (cf.~\eqref{eq:g-hat} below, for example).

\begin{thm}
\label{t:FK-cc}
Assume $A_{Q,W}$ is Noetherian, and let $\frobcat=\frobcat_{Q,W}$ with cluster-tilting object $T=eA_{Q,W}$. Then for any $X\in\frobcat$, the cluster character of $X$ with respect to $T$ is
\begin{equation}
\label{eq:FK-cc-pprin}
\FKcc{T}(X)=x^{g_X}\sum_{d\in\ZZ^{Q_0}}\lambda_dx^{v_d}(x^+)^{d}(x^-)^{\dimvec(GX)-d},
\end{equation}
for $g_X$, $v_d$ and $\lambda_d$ as above.
\end{thm}
\begin{proof}
Bearing in mind the identity \eqref{eq:stab-cc}, cf.\ \cite[Thm.~3.3(b)]{FuKel}, it is sufficient to compute the power of $x_i^\pm$ appearing in each term of the formula \eqref{eq:FK-cc} for $\FKcc{T}(X)$.

First we deal with the leading term $x^{\hat{g}_X}$. For $i\in Q_0$, the $i^\pm$-component of $\hat{g}_X$ is by definition $\langle FX,S_i^\pm\rangle$. This can be computed using the projective resolution \eqref{eq:F-proj-res} of $FX$, which shows that
\begin{align*}
\langle FX,S_i^\pm\rangle&=\dim\Hom_A(FR_X,S_i^\pm)-\dim\Hom_A(FK_X,S_i^\pm)\\
&=[R_X:P_i^\pm]-[K_X:P_i^\pm]
\end{align*}
for $R_X,K_X\in\add{T}$ as in \eqref{eq:index}; the second equality uses that $\dim\Hom_A(FR_X,S_i^\pm)$ counts the multiplicity of the projective $FP_i^\pm$ as a summand of the projective $FR_X$ (and similarly for $K_X$), and also the Yoneda lemma as in the proof of Lemma~\ref{l:index-sequence}. We may then use this lemma to see that
\begin{align*}
\langle FX,S_i^+\rangle&=0,\\
\langle FX,S_i^-\rangle&=\dim_\KK(e_iGX),
\end{align*}
so that
\begin{equation}
\label{eq:g-hat}
x^{\hat{g}_X}=x^{g_X}(x^-)^{\dim_\KK(GX)}.
\end{equation}

We calculate $x^{\hat{v}_d}$, for $d\in\ZZ^{\lift{Q}_0}$, similarly. The $i^\pm$-component is
\[\langle d,S_i^\pm\rangle=\sum_{j\in Q_0}d_j\langle S_j,S_i^\pm\rangle.\]
Now, arguing as for $\hat{g}_X$, we may use \eqref{eq:simp-proj-res} to compute that
\begin{align*}
\langle S_j,S_i^+\rangle&=\begin{cases}1,&i=j,\\0,&\text{otherwise},\end{cases}\\
\langle S_j,S_i^-\rangle&=\begin{cases}-1,&i=j,\\0,&\text{otherwise}.\end{cases}
\end{align*}
It then follows that
\begin{equation}
\label{eq:v-hat}
x^{\hat{v}_d}=x^{v_d}(x^+)^d(x^-)^{-d},
\end{equation}
and we obtain the desired result by substituting \eqref{eq:g-hat} and \eqref{eq:v-hat} back into \eqref{eq:FK-cc}.
\end{proof}

\begin{rem}
Using the interpretation of $\lambda_d$ as the Euler characteristic of a quiver Grassmannian, the non-zero terms of the sum in \eqref{eq:FK-cc-pprin} are indexed by (dimension vectors of) submodules of the $\jac{Q}{W}$-module $GX$. In the term indexed by such a submodule $N$, the exponent of $x^+$ is $\dimvec(N)$, and that of $x^-$ is $\dimvec(GX/N)$.
\end{rem}

\begin{cor}
The bijection between rigid indecomposable objects of $\frobcat_Q^+$ and cluster variables of the principal coefficient cluster algebra $\princlust{Q}$, as in Corollary~\ref{c:extriangulated}\ref{c:extriangulated-clustvars}, is given by $X\mapsto\FKcc{T}_+(X)$ where
\[\FKcc{T}_+(X)=x^{g_X}\sum_{d\in\ZZ^{Q_0}}\lambda_dx^{v_d}(x^+)^{d}.\]
\end{cor}
\begin{proof}
As in the proof of Corollary~\ref{c:extriangulated}, the bijection is given by specialising Fu--Keller's cluster character on $\frobcat_Q$ (where $X$ is also an object) at $x_i^-=1$ for all $i\in Q_0$. The result is then immediate from Theorem~\ref{t:FK-cc}.
\end{proof}

\begin{rem}
\label{r:borges-pierin}
The right-hand side of \eqref{eq:FK-cc-pprin} is precisely the `cluster character with coefficients' considered by Borges and Pierin \cite[Defn.~3.1]{borgescluster} for the cluster category of a Dynkin quiver. As a consequence, their expression satisfies the cluster character multiplication formula (see e.g.\ \cite[Thm.~2]{calderotriangulated2}), even after extending the setting to cluster categories of arbitrary acyclic quivers.
\end{rem}

\section{Examples}
\label{s:egs}

\begin{eg}
\label{e:a2eg}
Let $Q$ be an $\type{A}_2$ quiver, so, as computed in Example~\ref{e:a2},
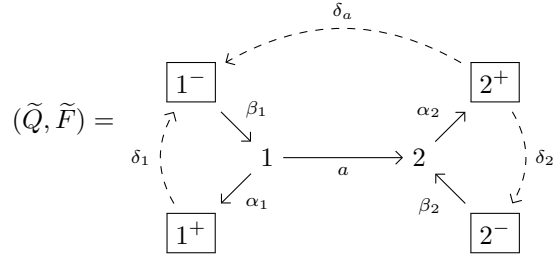
\[(\lift{Q},\lift{F})=\mathord{\begin{tikzpicture}[baseline={(current bounding box.center)},yscale=-1]
\node at (-1,0) (1) {$1$};
\node at (1,0) (2) {$2$};
\node at (-2,1) (1+) {$\boxed{1^+}$};
\node at (-2,-1) (1-) {$\boxed{1^-}$};
\node at (2,1) (2-) {$\boxed{2^-}$};
\node at (2,-1) (2+) {$\boxed{2^+}$};
\path[-angle 90,font=\scriptsize]
	(1) edge node[below] {$a$} (2)
	(1) edge node[below right] {$\alpha_1$} (1+)
	(1-) edge node[above right] {$\beta_1$} (1)
	(2) edge node[above left] {$\alpha_2$} (2+)
	(2-) edge node[below left] {$\beta_2$} (2);
\path[\frozen,-angle 90,font=\scriptsize]
	(1+) edge[bend right] node[left] {$\delta_1$} (1-)
	(2+) edge[bend right] node[right] {$\delta_2$} (2-)
	(2+) edge[bend left] node[above] {$\delta_a$} (1-);
\end{tikzpicture}}\]
and $\lift{W}=\beta_1\delta_1\alpha_1+\beta_2\delta_2\alpha_2-a\beta_1\delta_a\alpha_2$. Then $A_Q=\frjac{\lift{Q}}{\lift{F}}{\lift{W}}$, and its boundary algebra is $B_Q\iso\cpa{\KK}{\Lambda_Q}/\close{I}$ for
\[\Lambda_Q=\mathord{\begin{tikzpicture}[baseline={([yshift=-0.5ex]current bounding box.center)},scale=1.5]
\node at (0,0) (1) {$1$};
\node at (1,0) (2) {$2$};
\node at (2,0) (3) {$3$};
\node at (3,0) (4) {$4$};
\path[-angle 90,font=\scriptsize]
	(1) edge[bend left] node[above] {$\alpha$} (2)
	(2) edge[bend left] node[below] {$\dual{\alpha}$} (1)
	(2) edge[bend left] node[above] {$\beta$} (3)
	(3) edge[bend left] node[below] {$\dual{\beta}$} (2)
	(3) edge[bend left] node[above] {$\gamma$} (4)
	(4) edge[bend left] node[below] {$\dual{\gamma}$} (3);
\end{tikzpicture}}\]
and $I=\Span{\dual{\alpha}\alpha,\ \alpha\dual{\alpha}-\dual{\beta}\beta,\ \beta\dual{\beta}-\dual{\gamma}\gamma,\ \gamma\dual{\gamma},\ \beta\alpha,\ \gamma\beta,\ \dual{\alpha}\dual{\beta}\dual{\gamma}}$. Here we have relabelled the vertices by identifying the ordered sets $(1^+,1^-,2^+,2^-)$ and $(1,2,3,4)$. The Auslander--Reiten quiver of $\GP(B)$ is shown in Figure~\ref{fig:a2-AR}, where we identify the left and right sides of the picture so that the quiver is drawn on a Möbius band.
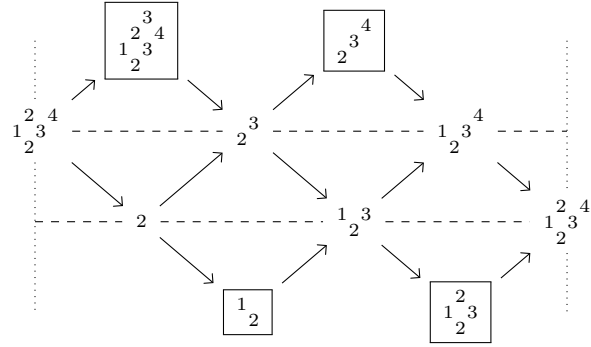
\begin{figure}[h]
\begin{tikzpicture}[yscale=1.2,xscale=1.4]
\node at (0,0) (radP3) {$\begin{radfilt}&2&&4\\1&&3\\&2\end{radfilt}$};
\node at (1,1) (P3) {$\boxed{\begin{radfilt}&&3\\&2&&4\\1&&3\\&2\end{radfilt}}$};
\node at (1,-1) (S2) {$\begin{radfilt}2\end{radfilt}$};
\node at (2,0) (radP4) {$\begin{radfilt}&3\\2\end{radfilt}$};
\node at (2,-2) (P1) {$\boxed{\begin{radfilt}1\\&2\end{radfilt}}$};
\node at (3,1) (P4) {$\boxed{\begin{radfilt}&&4\\&3\\2\end{radfilt}}$};
\node at (3,-1) (radP2) {$\begin{radfilt}1&&3\\&2\end{radfilt}$};
\node at (4,0) (M) {$\begin{radfilt}&&&4\\1&&3\\&2\end{radfilt}$};
\node at (4,-2) (P2) {$\boxed{\begin{radfilt}&2\\1&&3\\&2\end{radfilt}}$};
\node at (5,-1) (radP3') {$\begin{radfilt}&2&&4\\1&&3\\&2\end{radfilt}$};
\path[-angle 90]
	(radP3) edge (P3)
	(radP3) edge (S2)
	(P3) edge (radP4)
	(S2) edge (radP4)
	(S2) edge (P1)
	(radP4) edge (P4)
	(radP4) edge (radP2)
	(P1) edge (radP2)
	(P4) edge (M)
	(radP2) edge (M)
	(radP2) edge (P2)
	(M) edge (radP3')
	(P2) edge (radP3');
\path[dashed]
	(0,-1) edge (S2)
	(radP3) edge (radP4)
	(S2) edge (radP2)
	(radP4) edge (M)
	(radP2) edge (radP3')
	(M) edge (5,0);
\path[dotted]
	(0,1) edge (radP3)
	(radP3) edge (0,-2)
	(5,1) edge (radP3')
	(radP3') edge (5,-2);
\end{tikzpicture}
\caption{\label{fig:a2-AR}The Auslander--Reiten quiver of $\GP(B_Q)$ for $Q$ of type $\type{A}_2$.}
\end{figure}
To calculate the objects of $\GP(B_Q)$ it is useful to observe that, in this example, $B_Q$ is $1$-Iwanaga--Gorenstein, and so $\GP(B_Q)=\Sub(B_Q)$. The stable category $\stabGP(B_Q)$ is the cluster category of type $\type{A}_2$, as expected. The cluster tilting object
\[T=\begin{radfilt}&3\\2\end{radfilt}\oplus\begin{radfilt}1&&3\\&2\end{radfilt}\oplus B_Q\]
of $\GP(B_Q)$ has endomorphism algebra $A_Q$, and corresponds to the initial seed of the cluster algebra with polarised principal coefficients associated to our initial $\type{A}_2$ quiver $Q$.
\end{eg}

\begin{eg}
\label{e:a3-eg}
Let $Q$ be a linearly oriented quiver of type $\mathsf{A}_3$. We may then compute
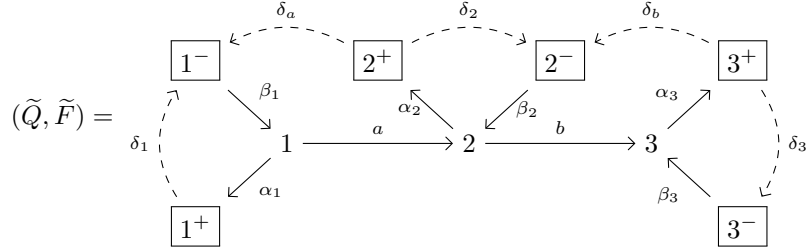
\[(\lift{Q},\lift{F})=\mathord{\begin{tikzpicture}[baseline={(current bounding box.center)},xscale=1.2,yscale=-1.1]
\node at (-1,0) (1) {$1$};
\node at (1,0) (2) {$2$};
\node at (3,0) (3) {$3$};
\node at (-2,1) (1+) {$\boxed{1^+}$};
\node at (-2,-1) (1-) {$\boxed{1^-}$};
\node at (0,-1) (2+) {$\boxed{2^+}$};
\node at (2,-1) (2-) {$\boxed{2^-}$};
\node at (4,1) (3-) {$\boxed{3^-}$};
\node at (4,-1) (3+) {$\boxed{3^+}$};
\path[-angle 90,font=\scriptsize]
	(1) edge node[above] {$a$} (2)
	(2) edge node[above] {$b$} (3)
	(1) edge node[below right] {$\alpha_1$} (1+)
	(1-) edge node[above right] {$\beta_1$} (1)
	(2) edge node[left] {$\alpha_2$} (2+)
	(2-) edge node[right] {$\beta_2$} (2)
	(3) edge node[above left] {$\alpha_3$} (3+)
	(3-) edge node[below left] {$\beta_3$} (3);
\path[\frozen,-angle 90,font=\scriptsize]
	(1+) edge[bend right] node[left] {$\delta_1$} (1-)
	(2+) edge[bend right] node[above] {$\delta_2$} (2-)
	(3+) edge[bend right] node[right] {$\delta_3$} (3-)
	(2+) edge[bend left] node[above] {$\delta_a$} (1-)
	(3+) edge[bend left] node[above] {$\delta_b$} (2-);
\end{tikzpicture}}\]
Relabelling vertices similarly to Example~\ref{e:a2eg}, the boundary algebra $B_Q$ has quiver
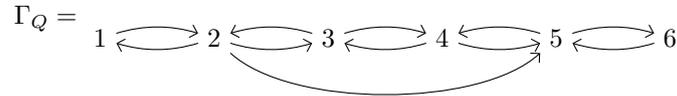
\[\Lambda_Q=\mathord{\begin{tikzpicture}[baseline={([yshift=3.6ex]current bounding box.center)},xscale=1.5,yscale=0.8]
\node at (0,0) (1) {$1$};
\node at (1,0) (2) {$2$};
\node at (2,0) (3) {$3$};
\node at (3,0) (4) {$4$};
\node at (4,0) (5) {$5$};
\node at (5,0) (6) {$6$};
\path[-angle 90,font=\scriptsize]
	(1) edge[bend left] (2)
	(2) edge[bend left] (1)
	(2) edge[bend right] (3)
	(3) edge[bend right] (2)
	(3) edge[bend left] (4)
	(4) edge[bend left] (3)
	(4) edge[bend right] (5)
	(5) edge[bend right] (4)
	(5) edge[bend left] (6)
	(6) edge[bend left] (5)
	(2) edge[bend right=60] (5);
\end{tikzpicture}}\]
as computed before in Example~\ref{e:a3-bdy-eg}. Explicit relations can be written down as in Section~\ref{s:bdy-algs}, but here we will simply give radical filtrations for the projective modules. Note that despite the `geographical' separation of $2$ and $5$ in these filtrations, the arrow $2\to 5$ always acts as a vector space isomorphism from the $1$-dimensional subspace of $e_2B_Q$ indicated by a $2$ in the filtration to the $1$-dimensional subspace of $e_5B_Q$ indicated by a $5$ in the row below, when this configuration occurs.
\begin{align*}
P_1&=\begin{radfilt}1\\&2\end{radfilt}&P_2&=\begin{radfilt}&2\\1&&3&&5\\&2&&4\end{radfilt}&P_3&=\begin{radfilt}&&3\\&2&&4\\1&&3&&5\\&2&&4\end{radfilt}\\[1ex]
P_4&=\begin{radfilt}&&4\\&3&&5\\2&&4\end{radfilt}&P_5&=\begin{radfilt}&&&5\\&&4&&6\\&3&&5\\2&&4\end{radfilt}&P_6&=\begin{radfilt}&&6\\&5\\4\end{radfilt}
\end{align*}

In this case the Gorenstein dimension of $B_Q$ is $2$; the indecomposable projective $P_2$ has injective dimension $2$, while all others have injective dimension $1$. The Auslander--Reiten quiver of $\GP(B_Q)$ is shown, again on a Möbius band, in Figure~\ref{fig:a3-AR}.
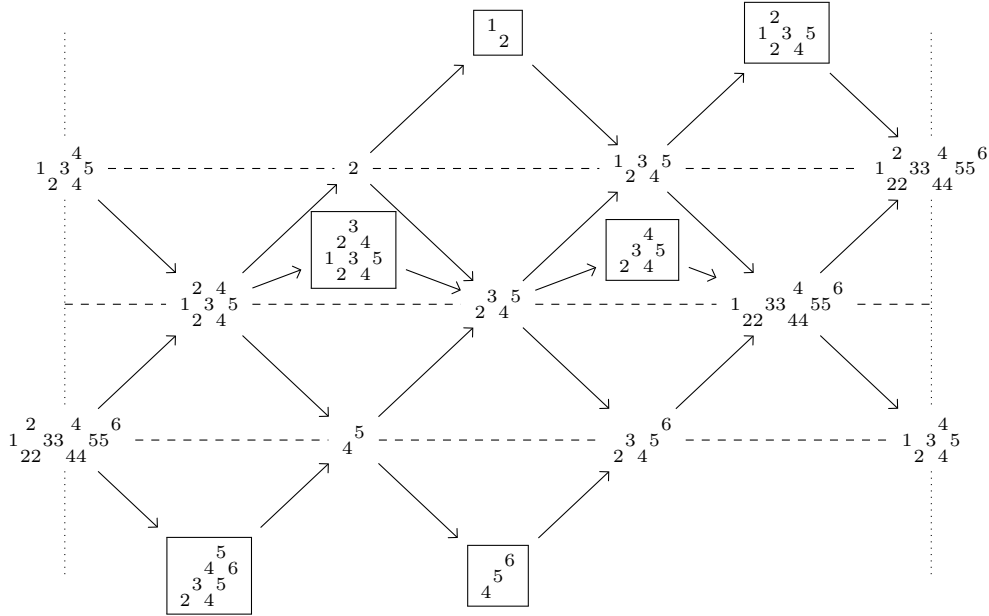
\begin{figure}[t]
\begin{tikzpicture}[yscale=1.8,xscale=1.9]
\node at (0,0) (M4) {$\begin{radfilt}&&&4\\1&&3&&5\\&2&&4\end{radfilt}$};
\node at (0,-2) (M9) {$\begin{radfilt}&2&&4&&6\\1&&33&&55\\&22&&44\end{radfilt}$};
\node at (1,-1) (M8) {$\begin{radfilt}&2&&4\\1&&3&&5\\&2&&4\end{radfilt}$};
\node at (1,-3) (P5) {$\boxed{\begin{radfilt}&&&5\\&&4&&6\\&3&&5\\2&&4\end{radfilt}}$}; 
\node at (2,0) (M7) {$\begin{radfilt}2\end{radfilt}$};
\node at (2,-0.6) (P3) {$\boxed{\begin{radfilt}&&3\\&2&&4\\1&&3&&5\\&2&&4\end{radfilt}}$}; 
\node at (2,-2) (M3) {$\begin{radfilt}&5\\4\end{radfilt}$};
\node at (3,1) (P1) {$\boxed{\begin{radfilt}1\\&2\end{radfilt}}$}; 
\node at (3,-1) (M2) {$\begin{radfilt}&3&&5\\2&&4\end{radfilt}$};
\node at (3,-3) (P6) {$\boxed{\begin{radfilt}&&6\\&5\\4\end{radfilt}}$}; 
\node at (4,0) (M1) {$\begin{radfilt}1&&3&&5\\&2&&4\end{radfilt}$};
\node at (4,-0.6) (P4) {$\boxed{\begin{radfilt}&&4\\&3&&5\\2&&4\end{radfilt}}$}; 
\node at (4,-2) (M6) {$\begin{radfilt}&&&&6\\&3&&5\\2&&4\end{radfilt}$};
\node at (5,1) (P2) {$\boxed{\begin{radfilt}&2\\1&&3&&5\\&2&&4\end{radfilt}}$}; 
\node at (5,-1) (M5) {$\begin{radfilt}&&&4&&6\\1&&33&&55\\&22&&44\end{radfilt}$};
\node at (6,0) (M9') {$\begin{radfilt}&2&&4&&6\\1&&33&&55\\&22&&44\end{radfilt}$};
\node at (6,-2) (M4') {$\begin{radfilt}&&&4\\1&&3&&5\\&2&&4\end{radfilt}$};

\path[-angle 90]
	(M4) edge (M8)
	(M9) edge (M8)
	(M9) edge (P5)
	(M8) edge (M7)
	(M8) edge (P3)
	(M8) edge (M3)
	(P5) edge (M3)
	(M7) edge (P1)
	(M7) edge (M2)
	(P3) edge (M2)
	(M3) edge (M2)
	(M3) edge (P6)
	(P1) edge (M1)
	(M2) edge (M1)
	(M2) edge (P4)
	(M2) edge (M6)
	(P6) edge (M6)
	(M1) edge (P2)
	(M1) edge (M5)
	(P4) edge (M5)
	(M6) edge (M5)
	(P2) edge (M9')
	(M5) edge (M9')
	(M5) edge (M4');
\path[dashed]
	(0,-1) edge (M8)
	(M4) edge (M7)
	(M9) edge (M3)
	(M8) edge (M2)
	(M7) edge (M1)
	(M3) edge (M6)
	(M2) edge (M5)
	(M1) edge (M9')
	(M6) edge (M4')
	(M5) edge (6,-1);
\path[dotted]
	(0,1) edge (M4)
	(M4) edge (M9)
	(M9) edge (0,-3)
	(6,1) edge (M9')
	(M9') edge (M4')
	(M4') edge (6,-3);
\end{tikzpicture}
\caption{\label{fig:a3-AR}The Auslander--Reiten quiver of $\GP(B_Q)$ for $Q$ linearly oriented of type $\type{A}_3$. In addition to the the usual mesh relations coming from Auslander--Reiten sequences, the length two path from $P_2$ to $P_5$ represents the zero map.}
\end{figure}
The initial cluster tilting object from Corollary~\ref{c:frobcat-constr} is
\[T=\begin{radfilt}1&&3&&5\\&2&&4\end{radfilt}\oplus\begin{radfilt}&3&&5\\2&&4\end{radfilt}\oplus\begin{radfilt}&5\\4\end{radfilt}\oplus B_Q.\]
\end{eg}

\begin{eg}
\label{e:3cycle-eg}
Applying our construction to the quiver with potential $(Q,W)$ from Example~\ref{e:3-cycle} in which $Q$ is a $3$-cycle (which we may do, since while $A$ is not finite-dimensional in this case, it is still Noetherian) yields, as observed in Example~\ref{e:3-cycle}, the Grassmannian cluster category for the Grassmannian $\mathrm{Gr}_{2}^{6}$ \cite{JKS}. This is a Hom-infinite category, and the Gorenstein projective $B$-modules are all infinite-dimensional. Representing these modules by Plücker labels as in \cite{JKS}, the Auslander--Reiten quiver of $\GP(B_{Q,W})$ is shown, on the now familiar Möbius band, in Figure~\ref{fig:3cycle-AR}.
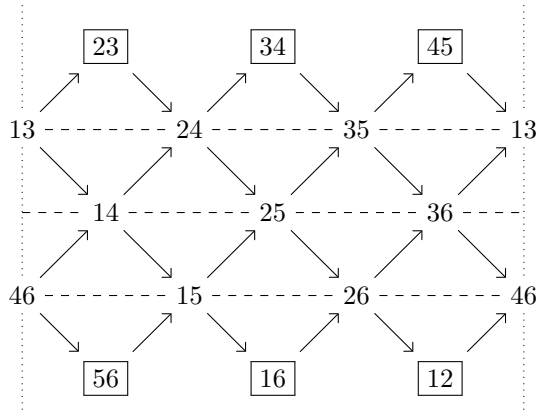
\begin{figure}[t]
\begin{tikzpicture}[scale=1.1]
\node at (0,1) (13) {$13$};
\node at (0,-1) (46) {$46$};
\node at (1,2) (23) {$\boxed{23}$};
\node at (1,0) (14) {$14$};
\node at (1,-2) (56) {$\boxed{56}$};
\node at (2,1) (24) {$24$};
\node at (2,-1) (15) {$15$};
\node at (3,2) (34) {$\boxed{34}$};
\node at (3,0) (25) {$25$};
\node at (3,-2) (16) {$\boxed{16}$};
\node at (4,1) (35) {$35$};
\node at (4,-1) (26) {$26$};
\node at (5,2) (45) {$\boxed{45}$};
\node at (5,0) (36) {$36$};
\node at (5,-2) (12) {$\boxed{12}$};
\node at (6,1) (13') {$13$};
\node at (6,-1) (46') {$46$};

\path[-angle 90]
	(13) edge (23)
	(13) edge (14)
	(46) edge (14)
	(46) edge (56)
	(23) edge (24)
	(14) edge (24)
	(14) edge (15)
	(56) edge (15)
	(24) edge (34)
	(24) edge (25)
	(15) edge (25)
	(15) edge (16)
	(34) edge (35)
	(25) edge (35)
	(25) edge (26)
	(16) edge (26)
	(35) edge (45)
	(35) edge (36)
	(26) edge (36)
	(26) edge (12)
	(45) edge (13')
	(36) edge (13')
	(36) edge (46')
	(12) edge (46');
\path[dashed]
	(0,0) edge (14)
	(13) edge (24)
	(46) edge (15)
	(14) edge (25)
	(24) edge (35)
	(15) edge (26)
	(25) edge (36)
	(35) edge (13')
	(26) edge (46')
	(36) edge (6,0);
\path[dotted]
	(0,2.5) edge (13)
	(13) edge (46)
	(46) edge (0,-2.5)
	(6,2.5) edge (13')
	(13') edge (46')
	(46') edge (6,-2.5);
\end{tikzpicture}
\caption{\label{fig:3cycle-AR}The Auslander--Reiten quiver of $\GP(B_{Q,W})$, where $(Q,W)$ is a $3$-cycle and its usual potential, shown as the Grassmannian cluster category $\CM(B_{2,6})$.}
\end{figure}
In this case, the quiver of the endomorphism algebra of the object
\[13\dsum 15\dsum 35\dsum B_{Q,W}\]
is $\lift{Q}$, as is the quiver of the (isomorphic) endomorphism algebra of $24\dsum 26\dsum 46\dsum B_{Q,W}$. We note that the stable category $\stabGP(B_{Q,W})$ is equivalent to the cluster category $\clustcat{Q,W}\simeq\clustcat{Q'}$ where $Q'$ is any orientation of the Dynkin diagram $\type{A}_3$. Thus all of the conclusions of Theorem~\ref{t:main-cat-thm} (replacing $\frobcat_Q$ by $\frobcat_{Q,W}$ and $\clustcat{Q}$ by $\clustcat{Q,W}$) still hold in this case, despite the failure of acyclicity.
\end{eg}

\section*{Acknowledgements}
We thank Mikhail Gorsky, Alfredo Nájera Chávez, Hiroyuki Nakaoka, Yann Palu and Salvatore Stella for useful conversations and for providing references. We are particularly grateful to Bernhard Keller for pointing out the role of pseudocompactness in the results of Section~\ref{s:frjacCY}, and also thank the anonymous referee for suggesting various improvements.

Financial support during the initial writing of the paper was provided by the Max-Planck-Gesellschaft. The paper was revised during a stay at the Isaac Newton Institute for Mathematical Sciences, Cambridge, supported by EPSRC grant EP/R014604/1, during the programme \emph{Representation theory and cluster algebras}, and we thank the institute for their support and hospitality. The author was further supported by the EPSRC postdoctoral fellowship grant EP/T001771/1 during this period.

\defbibheading{bibliography}[\refname]{\section*{#1}}
\printbibliography
\end{document}

%% file: options-std.tex
\providecommand{\lang}{UKenglish}
\providecommand{\geomoptions}{hmargin=3cm,vmargin=3.5cm}

\usepackage[\lang]{babel}
\usepackage[utf8]{inputenc}
\usepackage{lmodern}
\usepackage[T1]{fontenc}
\usepackage[\geomoptions]{geometry}
\usepackage{amsmath}
\usepackage{amssymb}
\usepackage{amsfonts}
\usepackage{amsthm}
\usepackage{mathrsfs}
\usepackage{booktabs}
\usepackage{microtype}

\usepackage[usenames,dvipsnames,svgnames]{xcolor}
\usepackage[pdflang=en-UK, colorlinks, urlcolor=Navy, linkcolor=Navy, citecolor=Navy]{hyperref}
\usepackage{color}
\usepackage{tikz}
	\usetikzlibrary{calc}
	\usetikzlibrary{arrows,decorations.markings}
\usepackage{tikz-cd}
\usetikzlibrary{arrows}
\pgfarrowsdeclarecombine{twohead}{twohead}
{angle 90}{angle 90}{angle 90}{angle 90}
\usepackage{microtype}

\usepackage{todonotes}

\usepackage{enumitem}
\setenumerate{label=(\alph*)}

\usepackage{subcaption}


\input{\level macros.tex}

%% file: macros.tex


\newcommand{\der}[1]{\partial_{#1}}
	
	\newcommand{\rightder}[1]{\partial^r_{#1}}

\newcommand{\Jacrad}[1]{J(#1)}
\newcommand{\Kdual}{\mathrm{D}}
\newcommand{\op}{\mathrm{op}}

\newcommand{\rad}[1][]{\operatorname{rad}^{#1}}
\newenvironment{radfilt}
  {\renewcommand\thickspace{\kern.05em}\smallmatrix}
  {\endsmallmatrix}

\newcommand{\simp}[1]{S_{#1}}
\newcommand{\Span}[1]{\langle#1\rangle}

\newcommand{\donothing}[1]{}


\newcommand{\dual}[1]{#1^*}
\newcommand{\dsum}{\mathbin{\oplus}}
	\newcommand{\bigdsum}{\bigoplus}
\newcommand{\End}{\operatorname{End}}
	\newcommand{\stabEnd}{\operatorname{\underline{End}}}
\newcommand{\Ext}{\operatorname{Ext}}
\newcommand{\Hom}{\operatorname{Hom}}
	\newcommand{\RHom}{\operatorname{\mathbf{R}Hom}}
	\newcommand{\stabHom}{\operatorname{\underline{Hom}}}
\newcommand{\Kgp}{\mathrm{K}}
\newcommand{\im}{\operatorname{im}}

\newcommand{\invlim}{\varprojlim}
\newcommand{\iso}{\cong}
\newcommand{\isoto}{\stackrel{\sim}{\to}}
\newcommand{\map}[1]{\stackrel{#1}{\longrightarrow}}
\newcommand{\Ob}{\operatorname{Ob}}
\newcommand{\tensor}{\mathbin{\otimes}}





\newcommand{\gldim}{\operatorname{gl.dim}}

\newcommand{\pdim}{\operatorname{p.dim}}

\newcommand{\bdd}{\mathrm{b}}

\newcommand{\EE}{\mathbb{E}}
\newcommand{\KK}{\mathbb{K}}

\newcommand{\QQ}{\mathbb{Q}}

\newcommand{\ZZ}{\mathbb{Z}}


\newcommand{\dimvec}{\operatorname{\underline{dim}}}
\newcommand{\double}[1]{\overline{#1}}
\newcommand{\head}[1]{h#1}
\newcommand{\Head}[2][]{H^{#1}_{#2}}
\newcommand{\idemp}[1]{e_{#1}}

\newcommand{\tail}[1]{t#1}
\newcommand{\Tail}[2][]{T^{#1}_{#2}}


\newcommand{\set}[1]{\{#1\}}


\newcommand{\clustalg}[1]{\mathscr{A}_{#1}}
\newcommand{\cpa}[2]{#1\langle\hspace{-0.1em}\langle #2\rangle\hspace{-0.1em}\rangle}
\newcommand{\Endalg}[2]{\End_{#1}(#2)^{\op}}
	\newcommand{\stabEndalg}[2]{\stabEnd_{#1}(#2)^{\op}}
\newcommand{\env}[1]{{#1}^{\mathrm{e}}}
\newcommand{\jac}[2]{J(#1,#2)}
	\newcommand{\frjac}[3]{J(#1,#2,#3)}

\newcommand{\preproj}{\Pi}
	\newcommand{\dgpreproj}{\mathbf{\Pi}}


	\newcommand{\add}{\operatorname{add}}
\newcommand{\clustcat}[1]{\mathcal{C}_{#1}}
\newcommand{\CM}{\operatorname{CM}}
\newcommand{\dcat}[1][]{\mathcal{D}^{#1}}
	\newcommand{\bdcat}{\mathcal{D}^{\bdd}}

\newcommand{\fd}{\operatorname{fd}}
\newcommand{\frobcat}{\mathcal{E}}
\newcommand{\GP}{\operatorname{GP}}
	\newcommand{\stabGP}{\operatorname{\underline{GP}}}

\newcommand{\Modcat}{\operatorname{Mod}}
	\newcommand{\fgmod}{\operatorname{mod}}
	
\newcommand{\per}{\operatorname{per}}
\newcommand{\projcat}{\operatorname{proj}}

\newcommand{\Sub}{\operatorname{Sub}}
	
\newcommand{\stab}[1]{\underline{#1}}


\newcommand{\Grass}[2]{\mathrm{Gr}_{#1}^{#2}}

\newcommand{\blank}{\mathbin{\char123}}

\renewcommand{\epsilon}{\varepsilon}
\renewcommand{\geq}{\geqslant}
\renewcommand{\leq}{\leqslant}
\newcommand{\type}{\mathsf}


\newcommand{\close}[1]{\overline{#1}}

%% file: amsthm-std.tex
\theoremstyle{plain}
\newtheorem{thm}{Theorem}[section]
\newtheorem{thm*}{Theorem}
\newtheorem{lem}[thm]{Lemma}
\newtheorem{cor}[thm]{Corollary}
\newtheorem{cor*}{Corollary}
\newtheorem{prop}[thm]{Proposition}
\newtheorem{conj}{Conjecture}

\theoremstyle{definition}
\newtheorem{defn}[thm]{Definition}
\newtheorem{eg}[thm]{Example}

\theoremstyle{remark}
\newtheorem{rem}[thm]{Remark}
\newtheorem*{nb}{Note}

\numberwithin{equation}{section}

%% file: biblatex-std.tex
\usepackage[backend=bibtex,citestyle=numeric-comp,bibstyle=numeric,maxbibnames=99,giveninits=true,doi=false,isbn=false,url=false,eprint=true]{biblatex}

\renewbibmacro{in:}{%
  \ifentrytype{article}{}{\printtext{\bibstring{in}\intitlepunct}}}

\DeclareFieldFormat[article,inbook,incollection,inproceedings,patent]{title}{#1}
\DeclareFieldFormat[thesis,unpublished]{title}{{\it #1}}

\DeclareFieldFormat[online]{date}{Preprint (#1)}

\DeclareFieldFormat[article]{volume}{\mkbibbold{#1}}
\renewbibmacro*{volume+number+eid}{%
  \printfield{volume}%
  \setunit*{\addnbspace}
  \printfield{number}%
  \setunit{\addcomma\space}%
  \printfield{eid}}
\DeclareFieldFormat[article]{number}{}

\DeclareFieldFormat[book,incollection]{number}{Vol.~#1}
\renewbibmacro*{series+number}{%
  \iffieldundef{series}{}
    {\printfield{series}%
      \iffieldundef{number}{}{\setunit*{\addcomma\addspace}%
       \printfield{number}%
       \newunit}}}

\renewbibmacro*{publisher+location+date}{%
  \printlist{publisher}%
  \setunit*{\addcomma\space}%
  \usebibmacro{date}%
  \newunit}
  
\DeclareFieldFormat{eprint:arxiv}{%
\ifentrytype{online}
  {\ifhyperref
    {\href{http://arxiv.org/abs/#1}{\nolinkurl{arXiv:#1}}}
    {\nolinkurl{arXiv:#1}}
   \iffieldundef{eprintclass}
    {}
    {{\tt \mkbibbrackets{\thefield{eprintclass}}}}}
  {\iffieldundef{eprintclass}
    {\mkbibparens{\ifhyperref
    {\href{http://arxiv.org/abs/#1}{\nolinkurl{arXiv:#1}}}
    {\nolinkurl{arXiv:#1}}}}
    {\mkbibparens{\ifhyperref
    {\href{http://arxiv.org/abs/#1}{\nolinkurl{arXiv:#1}}}
    {\nolinkurl{arXiv:#1}}
    {\tt \mkbibbrackets{\thefield{eprintclass}}}}}}}
\addbibresource{\level/mainbib.bib}